\documentclass[oneside,11pt]{amsart}

\usepackage{float}
\usepackage{amscd,amssymb}
\usepackage[mathscr]{eucal}
\usepackage[all]{xy}
\usepackage{mathrsfs}
\usepackage{amsfonts}
\usepackage{amsmath}
\usepackage{amsthm}
\usepackage{latexsym}
\usepackage[all]{xy}
\usepackage{phonetic}
\usepackage{todonotes}

\title{The Frobenius morphism on flag varieties, I}

\author{Alexander Samokhin}

\address{Math. Institut, Heinrich-Heine-Universit\"at, D-40204 D\"usseldorf, Germany}
\address{\it and}
\address{Institute for Information Transmission Problems, Moscow,  Russia}
\email{alexander.samokhin@gmail.com}

\newcommand{\Oo}{\mathcal O}
\newcommand{\Uu}{\mathcal U}
\newcommand{\Ss}{\mathcal S}

\newcommand{\Pp}{\mathbb P}

\newcommand{\Ff}{\mathcal F}
\newcommand{\G}{\mathcal G}
\newcommand{\Ee}{\mathcal E}

\newcommand{\Ll}{\mathcal L}

\newcommand{\D}{\mathcal D}
\newcommand{\C}{\mathcal C}
\newcommand{\T}{\mathcal T}
\newcommand{\Kk}{\mathcal K}
\newcommand{\A}{\mathcal A}
\newcommand{\B}{\mathcal B}

\newcommand*{\RHom}{\mathop{\mathrm RHom}\nolimits}
\newcommand*{\Hom}{\mathop{\mathrm Hom}\nolimits}
\newcommand*{\Dd}{\mathop{\mathrm D\kern0pt}\nolimits}
\newcommand*{\DD}{\mathop{\mathbb D\kern0pt}\nolimits}

\newcommand*{\Ext}{\mathop{\mathrm Ext}\nolimits}

\newtheorem{theorem}{Theorem}[section]
\newtheorem{corollary}{Corollary}[section]
\newtheorem{lemma}{Lemma}[section]
\newtheorem{proposition}{Proposition}[section]

\newtheorem{remark}{Remark}[section]
\newtheorem{definition}{Definition}[section]	
\newtheorem{example}{Example}[section]
\numberwithin{equation}{section}

\long\def\comment#1{}

\begin{document}
	
\maketitle

\begin{abstract}
In this paper, given a semisimple algebraic group $\bf G$ of rank 2, 
we construct a special semiorthogonal decomposition in the derived category of coherent sheaves on the flag variety ${\bf G}/{\bf B}$. These decompositions are defined over the localization ${\mathbb Z}_{\rm S}$, where $\rm S$ is the set of bad primes for $\bf G$, while their block structure is compatible with the Bruhat order on Schubert varieties. The non--standard $t$--structures on $\Dd ^b({\bf G}/{\bf B})$ defined by these decompositions are self--dual with respect to the duality ${\mathcal RHom}_{{\bf G}/{\bf B}}(-,\omega _{{\bf G}/{\bf B}}^{\frac{1}{2}})$ given by the square root of the canonical sheaf of ${\bf G}/{\bf B}$. For the groups of classical type, this allows to construct an explicit decomposition of the higher Frobenii pushforward bundles ${\sf F}^n_{\ast}\Oo _{{\bf G}/{\bf B}}$ into a direct sum of indecomposable bundles. When $p>h$, the Coxeer number of the corresponding group, this set of indecomposable bundles forms a full exceptional collection in $\Dd ^b({\bf G}/{\bf B})$ defined over ${\mathbb Z}_{\rm S}$.
\end{abstract}

\vspace*{0.5cm}

\section{Introduction}

\vspace*{0.5cm}

Fix an algebraically closed field $k$ of characteristic $p>0$. Consider a simply connected semisimple algebraic group $\bf G$ over $k$,
and let ${\bf G}_n$ be the $n$--th Frobenius kernel.
Our main interest is in understanding the vector bundle ${\sf F}^n_{\ast}{\Oo}_{\bf G/B}$ on the flag variety ${\bf G}/{\bf B}$, where ${\sf F}_n: {\bf G/B}\rightarrow {\bf G/B}$ is the ($n$--th power of) Frobenius morphism. The bundle in question is homogeneous, and one would like to understand its ${\bf G}_n{\bf B}$--structure, where ${\bf G}_n{\bf B}$ is the pullback of ${\bf B}$ via the Frobenius morphism on $\bf G$. However, such a description seems to be missing so far in general (cf. \cite[Appendix]{And}: "\dots a good description of the actual ${\bf B}$--structure of this representation seems to be hard to obtain, although such a description certainly would be very useful).

On the other hand, there are sheaves of differential operators with divided powers ${\mathcal D}_{\bf G/B}$ that are related to ${{\sf F}_n}_{\ast}{\Oo}_{\bf G/B}$ via the formula ${\mathcal D}_{\bf G/B} = \cup _{n}{\mathcal End}({{\sf F}_n}_{\ast}{\Oo}_{\bf G/B})$.
Higher cohomology vanishing for sheaves of differential operators 
is essential for localization type theorems, of which the Beilinson--Bernstein localization \cite{BB} is the prototype. It is known that, unlike the characteristic zero case, localization of ${\mathcal D}$--modules  in characteristic $p$
does not hold in general \cite{KaLau}. However, one hopes for a weaker, but still sensible statement. Namely, consider the sheaf 
${\mathcal D}^{(1)}_{\bf G/B}= {\mathcal End}({\sf F}_{\ast}{\Oo}_{\bf G/B})$ (the first term of the $p$--filtration on ${\mathcal D}_{\bf G/B}$ on ${\bf G}/{\bf B}$). The category of 
sheaves of modules over ${\mathcal D}^{(1)}_{\bf G/B}$ is precisely that of ${\mathcal D}_{\bf G/B}$--modules with vanishing $p$--curvature, which is 
equivalent to the category of coherent sheaves via Cartier's equivalence. The question is then whether localization theorem holds for the category ${\mathcal D}^{(1)}_{\bf G/B}$--mod (see \cite{HKR}).

Representation--theoretic approach to ${\bf G}_1{\bf B}$--structure of ${\sf F}_{\ast}{\Oo}_{\bf G/B}$ is worked out in 
\cite{KanYe} and is based on the study of ${\bf G}_1{\bf T}$ (or ${\bf G}_1{\bf B}$)--structure of the induced ${\bf G}_1{\bf B}$--module ${\hat \nabla}(0): = {\rm Ind}_{\bf B}^{{\bf G}_1{\bf B}}\chi _{0}$ (the so-called Humphreys--Verma module). It is generally believed 
(see {\it loc.cit.}) that indecomposable subfactors of ${\sf F}_{\ast}{\Oo}_{\bf G/B}$, which are defined over $\mathbb Z$,
should form a full exceptional collection in the derived category $\Dd ^b({\rm Coh}({\bf G/B}))$. When it holds, it can be seen as a refinement
of the above equivalence for ${\mathcal D}^{(1)}_{\bf G/B}$.
Such a decomposition was obtained for projective spaces and the flag variety ${\bf SL}_3/{\bf B}$ in \cite{HKR}, and for smooth quadrics in \cite{Ach} and \cite{Lan}.

In this paper, we suggest an approach to decomposing the bundle ${{\sf F}_n}_{\ast}{\Oo}_{\bf G/B}$ for $n\in \mathbb N$ that is based on the existence of an appropriate semiorthogonal decomposition of the derived category of coherent sheaves $\Dd ^b({\bf G/B})$. In a nutshell, the idea is as follows.
Given a smooth algebraic variety $X$ of dimension $m$ over $k$, fix an $n\in \mathbb N$ and consider its derived category ${\rm D}^b(X)$. Assume that it admits a semiorthogonal decomposition ${\rm D}^b(X) = \langle \A _{-m},\A _{-n+1},\dots , \A _0\rangle$ with the following property: each admissible category $\A _{-i}$ is generated by an exceptional collection of vector bundles $\{\Ee ^i_l\}_{l=1}^{l=k_i}\in {\rm D}^b(X)$ that are pairwise orthogonal to each other, so that the category $\A _{-i}$ is equivalent to a direct sum of copies ${\rm D}^b({\sf Vect}-k)$. Further, assume that ${\rm H}^j(X,{\sf F}_n^{\ast}\Ee ^i_l)=0$ for $j\neq i$ and for each exceptional bundle $\Ee ^i_l\in \A _{-i}, i = 1,\dots ,m$.

We show in Theorem \ref{th:Frobdecomposclaim} that under these assumptions the bundle ${\sf F}_{\ast}\Oo _X$ decomposes into a direct sum of vector bundles. Indecomposable summands of this decomposition form the so-called right dual semiorthogonal decomposition with respect to 
$\langle \A _{-n},\A _{-n+1},\dots , \A _0\rangle$. Back to flag varieties ${\bf G/B}$, the task of decomposing the bundle ${{\sf F}_n}_{\ast}{\Oo}_{\bf G/B}$ is therefore broken up into two parts: one has first to find a semiorthogonal decomposition of ${\rm D}^b({\bf G/B})$ with the above properties; once such a decomposition has been found, one then calculates the right dual collection, and irreducible summands of ${\sf F}_{\ast}\Oo _{\bf G/B}$ are precisely terms of the latter collection. The multiplicity spaces at each indecomposable summand are identified with cohomology groups ${\rm H}^i(X,{\sf F}^{\ast}\Ee ^i_l)$ that are non--trivial at a unique cohomological degree.

One motivation for the present paper comes from the derived localization theorem of \cite{BMR}. It  implies, in particular, that for a regular weight $\chi$ (that is, for a weight having trivial stabilizer with respect to the dot--action of the (affine) Weyl group)
the bundle ${\sf F}_{\ast}\Ll _{\chi}$ is a generator in the derived category $\Dd ^b({\bf G}/{\bf B})$; in other words, there are sufficiently many  indecomposable summands of 
${\sf F}_{\ast}\Ll _{\chi}$ to generate the whole derived category $\Dd ^b({\bf G}/{\bf B})$. Knowing indecomposable summands of these bundles (e.g., for $p$--restricted weights) may clarify, in particular, cohomology vanishing patterns of line bundles on ${\bf G}/{\bf B}$. 
This paper also continues the study of sheaves of differential operators on flag varieties
in characteristic $p$ that we begun in a series of papers  \cite{Sam} -- \cite{Vanth}.

The paper is organized as follows.  We first recall in Sections \ref{sec:Cohomology_on_G/B} and \ref{sec:SOD_mutations} the necessary definitions and statements on cohomology of homogeneous vector bundles on flag varieties in characteristic $p$, and on semiorthogonal decompositions and mutations in derived categories. Upon stating the necessary prerequisites, Theorem \ref{th:Frobdecomposclaim} follows almost immediately. Next, based on the previous theorem, we work out in Section \ref{sec:P^n_and_quadrics} two classical examples: that of projective spaces, and that of smooth quadrics. In Section \ref{sec:Frobeniusdecompositions} we show how the above method works for classical groups of rank two. 

\vspace{0.2cm}

\subsection*{\bf Acknowledgements}
 I am very much indebted to Alexander Kuznetsov for many useful conversations over the years. 
The first drafts of this paper have been written in 2012--13 while the author benefited from generous supports of the SFB/Transregio 45 at the University of Mainz, and of the SFB 878 at the University of M\"unster. It is a great pleasure to thank Manuel Blickle and Christopher Deninger for their invitations and their interest in this work. The author gratefully acknowledges support from the strategic research fund of the Heinrich-Heine-Universit\"at D\"usseldorf (grant SFF F-2015/946-8).

\vspace{0.2cm}

\subsection*{Notation}
Throughout we fix a perfect field $k$ of characteristic $p>0$.
Given a split semisimple simply connected algebraic group $\bf G$ over $k$, let $\bf T$ denote a maximal torus of $\bf G$, and let ${\bf T}\subset {\bf B}$ be a Borel subgroup containing $\bf T$. The flag variety of Borel subgroups in $\bf G$ is denoted ${\bf G/B}$.
Denote ${\rm X}({\bf T})$ the weight lattice, and let $\rm R$ and  $\rm R ^{\vee}$ denote the root and coroot lattices, respectively. Let $\rm S$ be the set of simple roots relative to the choice of a Borel subgroup than contains $\bf T$. 
The Weyl group ${\mathcal W}={\rm N}({\bf T})/{\bf T}$ acts on $X({\bf T})$ via the dot--action: if $w\in {\mathcal W}$, and 
$\lambda \in {\rm X}({\bf T})$, then $w\cdot \lambda = w(\lambda + \rho) - \rho$. A parabolic subgroup of $\bf G$ is denoted by $\bf P$. For a simple root $\alpha \in {\rm S}$, denote ${\bf P}_{\alpha}\subset \bf G$ the corresponding minimal parabolic subgroup.  Given a weight $\lambda \in  {\rm X}({\bf T})$, denote $\Ll _{\lambda}$ the corresponding line bundle on ${\bf G}/{\bf B}$.  The half sum of the positive roots (the sum of fundamental weights) is denoted by $\rho$. Given a dominant 
weight $\lambda \in  {\rm X}({\bf T})$, we denote $\nabla _{\lambda}$ (resp., $\Delta _{\lambda}$) the induced module ${\rm Ind}_{\bf B}^{\bf G}\lambda$ (resp., the Weyl module), and the simple module with the highest weight $\lambda$ is denoted ${\sf L}_{\lambda}$. Given a variety $X$ and $n\in \mathbb N$, denote ${\sf F}_n$ the $n$--th iteration of the absolute  Frobenius morphism ${\sf F}_n: X\rightarrow X$. For a vector space ${\sf V}$ over $k$ its $n$-th Frobenius twist ${\sf F}_n^{\ast}{\sf V}$ is denoted ${\sf V}^{[n]}$. All the functors are supposed to be derived, i.e., given a morphism $f:X\rightarrow Y$ between two schemes, we write $f_{\ast},f^{\ast}$ for the corresponding derived functors of push--forwards and pull--backs.

\vspace*{0.5cm}

\section{Cohomology of line bundles on ${\bf G}/{\bf B}$}\label{sec:Cohomology_on_G/B}

\vspace*{0.5cm}

\subsection{\bf Flag varieties of Chevalley groups over ${\mathbb Z}$}

Let ${\mathbb G}\rightarrow \mathbb Z$ be a semisimple Chevalley group scheme (a smooth affine group scheme over ${\rm Spec}(\mathbb Z)$ whose geometric fibres are connected semisimple algebraic groups), and ${\mathbb G}/{\mathbb B}\rightarrow \mathbb Z$ be the corresponding Chevalley flag scheme (resp., ${\mathbb P}\subset \mathbb G$ the corresponding parabolic subgroup scheme over ${\mathbb Z}$). 
Then ${\mathbb G/\mathbb P}\rightarrow {\rm Spec}({\mathbb Z})$ is flat and the line bundle $\Ll$ on ${\bf G/P}$ also comes from a line bundle $\mathbb L$ on ${\mathbb G/\mathbb P}$. Let $k$ be a field of arbitrary characteristic, and ${\bf G/B}\rightarrow {\rm Spec}(k)$ be the flag variety obtained by base change along ${\rm Spec}(k)\rightarrow {\rm Spec}(\mathbb Z)$.

\vspace*{0.2cm}

\subsection{\bf Bott's vanishing theorem}\label{subsec:cohlinbunflags}

We recall first the classical Bott's theorem (see \cite{Dem}). Let ${\mathbb G}\rightarrow \mathbb Z$ be a semisimple Chevalley group scheme as above. Assume given a weight $\chi \in X({\mathbb T})$, and let $\Ll _{\chi}$ be the corresponding line bundle on ${\mathbb G}/{\mathbb B}$. The weight $\chi$ is called {\it singular}, if it lies on a wall of some Weyl chamber defined by $\langle -, \alpha ^{\vee}\rangle =0$ for some coroot $\alpha ^{\vee}\in {\rm R}^{\vee}$. Weights, which are not singular, are called {\it regular}. Let $k$ be a field of characteristic zero, and ${\bf G}/{\bf B}\rightarrow {\rm Spec}(k)$ the corresponding flag variety over $k$. The weight $\chi \in X({\bf T})$ defines a line bundle $\Ll _{\chi}$ on 
${\bf G}/{\bf B}$.

\begin{theorem}\cite[Theorem 2]{Dem}\label{th:Bott-Demazure_th}

\vspace*{0.2cm}

\begin{itemize}

\vspace*{0.2cm}

\item[(a)] If $\chi +\rho$ is singular, then ${\rm H}^i({\bf G}/{\bf B},\Ll _{\chi})= 0$ for all $i$.

\vspace*{0.2cm}

\item[(b)] If If $\chi + \rho$  is regular and dominant, then ${\rm H}^i({\bf G}/{\bf B},\Ll _{\chi}) = 0$ for $i>0$.

\vspace*{0.2cm}

\item[(c)]  If $\chi + \rho$  is regular, then ${\rm H}^i({\bf G}/{\bf B},\Ll _{\chi})\neq 0$ for the unique degree $i$, which is equal to $l(w)$. Here $l(w)$ is the length of an element of the Weyl group that takes $\chi$ to the dominant chamber, i.e. $w\cdot \chi \in X_{+}({\bf T})$. The cohomology group ${\rm H}^{l(w)}({\bf G}/{\bf B},\Ll _{\chi})$ is the irreducible $\bf G$--module of highest weight $w\cdot \chi$.

\end{itemize}

\end{theorem}


\subsection{Cohomology of line bundles} 

Some bits of Theorem \ref{th:Bott-Demazure_th} are still true over $\mathbb Z$.

\begin{proposition}\label{prop:acyclic_weights}
If a weight $\chi$ is such that $\langle \chi + \rho, \alpha ^{\vee}\rangle =0$ for some simple root $\alpha$, then the corresponding line bundle is acyclic. 
\end{proposition}

\begin{proof}
Indeed, Lemma from \cite[Section 2]{Dem} holds over fields of arbitrary characteristic (cf. also \cite[Lemma 1.1]{And}).
\end{proof}

Further, Kempf's vanishing theorem \cite[Part II, Chapter 5]{Jan} asserts that the statement of Theorem \ref{th:Bott-Demazure_th} remains true in characteristic $p$ if the weight in question is dominant:

\begin{theorem}[Kempf's vanishing theorem]\label{th:Kempf}
Let $\chi \in X(\bf T)$, i.e. $\langle \chi ,\alpha ^{\vee}\rangle \geq 0$ for all simple coroots $\alpha ^{\vee}$. Then ${\rm H}^i({\bf G}/{\bf B},\Ll _{\chi})=0$ for $i>0$.
\end{theorem}

Besides this, however, very little of Theorem \ref{th:Bott-Demazure_th} holds over $\mathbb Z$ \cite[Part II, Chapter 5]{Jan}. However, it still holds for weights lying in the interior of the bottom alcove in the  dominant chamber \cite[Theorem 2.3 and Corollary 2.4]{And}:

\begin{theorem}\label{th:AndthII}
If $\chi$ is a weight such that for a simple root $\alpha$ one has
$0\leq \langle \chi + \rho,\alpha ^{\vee}\rangle \leq p$ then 
\begin{equation}
{\rm H}^i({\bf G}/{\bf B},\Ll _{\chi}) = {\rm H}^{i+1}({\bf G}/{\bf
  B},\Ll _{s_{\alpha}\cdot \chi}).
\end{equation}
\end{theorem}


\subsubsection{First cohomology group of a line bundle}

\vspace*{0.2cm}

Let $\alpha$ be a simple root, and denote $s_{\alpha}$ a corresponding reflection in $\mathcal W$. One has $s_{\alpha}\cdot
\chi = s_{\alpha}(\chi)-\alpha$. There is a complete description \cite[Theorem 3.6]{An} 
of (non)--vanishing of the first cohomology group 
of a line bundle $\Ll _{\lambda}$.

\begin{theorem}\label{th:Andth}
${\rm H}^1({\bf G}/{\bf B},\Ll _{\chi})\neq 0$ if and only if
there exist a simple root $\alpha$ such that one of the following conditions is satisfied:

\begin{itemize}

\vspace*{0.2cm}

\item $-p\leq \langle \chi ,\alpha ^{\vee}\rangle \leq -2$ and
  $s_{\alpha}\cdot \chi = s_{\alpha}(\chi)-\alpha$ is dominant.
  
  \vspace*{0.2cm}

\item $\langle \chi ,\alpha ^{\vee}\rangle = -ap^n-1$ for some $a,n\in
  {\bf N}$ with $a<p$ and $s_{\alpha}(\chi)-\alpha$ is dominant.
  
  \vspace*{0.2cm}

\item $-(a+1)p^n\leq  \langle \chi ,\alpha ^{\vee}\rangle \leq
  -ap^n-2$ for some $a,n\in {\bf N}$ with $a<p$ and $\chi +
  ap^n\alpha$ is dominant.

\vspace*{0.2cm}  
  
\end{itemize}
\end{theorem}

One also has \cite[Corollary 3.2]{An}: 

\begin{theorem}\label{th:Andcor}
Let $\chi$ be a weight. If either $\langle \chi ,\alpha ^{\vee}\rangle
\geq -p$ or  $\langle \chi ,\alpha ^{\vee}\rangle = -ap^n-1$ for some
$a, n\in {\bf N}$ and $a<p$ then 
\begin{equation}
{\rm H}^i({\bf G}/{\bf B},\Ll _{\chi}) = {\rm H}^{i-1}({\bf G}/{\bf
  B},\Ll _{s_{\alpha}\cdot \chi}).
\end{equation}
\end{theorem}

\vspace*{0.5cm}

\section{Semiorthogonal decompositions, mutations, and exceptional collections}\label{sec:SOD_mutations}

\vspace*{0.5cm}

\subsection{Semiorthogonal decompositions}

Let $k$ be a field. Assume given a $k$--linear triangulated category $\D$, equipped with a shift functor $[1]\colon \D \rightarrow \D$.  For two
objects $A, B \in \D$ let $\Hom ^{\bullet}_{\D}(A,B)$ be
the graded $k$-vector space $\oplus _{i\in \mathbb Z}\Hom _{\D}(A,B[i])$. 
Let $\A \subset \D$ be a full triangulated subcategory,
that is a full subcategory of $\D$ which is closed under shifts.

The original source for most of the definitions and statements in this section is \cite{Bo}. 
We follow the expositions of  \cite[Section 2.1]{Efim} and \cite[Section 2.2]{Kuzrat}.

\begin{definition}\label{def:orthogonalcat}
The right orthogonal $\A ^{\perp}\subset \D$ is defined to be
the full subcategory

\vspace*{0.2cm}

\begin{equation}
\A^{\perp} = \{B \in \D \colon \Hom _{\D}(A,B) = 0 \}
\end{equation}

\vspace*{0.2cm}

\noindent for all $A \in \A$. The left orthogonal $^{\perp}\A$ is defined similarly. 

\end{definition}

\begin{definition}\label{def:admissible}
A full triangulated subcategory $\A$ of $\D$ is called
{\it right admissible} if the inclusion functor $\A\hookrightarrow \D$ has a right adjoint. Similarly, $\A$ is called {\it left
  admissible} if the inclusion functor has a left adjoint. Finally,
$\A$ is {\it admissible} if it is both right and
left admissible.
\end{definition}

If a full triangulated category $\A\subset \D$ is right admissible then every object $X\in \D$ fits into a distinguished triangle
  
\vspace*{0.2cm}
  
\begin{equation}
\dots \longrightarrow  Y\longrightarrow X\longrightarrow Z\longrightarrow Y[1]\rightarrow \dots
\end{equation}

\vspace*{0.2cm}

\noindent with $Y\in \A$ and $Z\in \A^{\perp}$. One then
says that there is a semiorthogonal decomposition of $\D$ into
the subcategories $(\A^{\perp}, \ \A)$. More generally,
assume given a sequence of full triangulated subcategories $\A _1,\dots,\A_n \subset \D$. Denote $\langle \A _1,\dots,\A_n\rangle$ the triangulated subcategory of $\D$ generated by $\A_1,\dots,\A_n$.

\begin{definition}\label{def:semdecomposition}
A sequence $(\A_1,\dots,\A_n)$ of admissible subcategories of
$\D$ is called {\it semiorthogonal} if $\A _i\subset \A_j^{\perp}$ for $1\leq i < j\leq n$,
and $\A_i\subset {^{\perp}\A_j}$ for $1\leq j < i\leq n$.
The sequence $(\A_1,\dots,\A_n)$ is called a {\it semiorthogonal
  decomposition} of $\D$ if $\langle \A_1, \dots, \A_n
\rangle^{\perp} = 0$, that is $\D = \langle \A_1,\dots,\A_n\rangle$.
\end{definition}

\vspace{0.2cm}

The above definition is equivalent to:

\begin{definition}\label{def:semdecomposition-filtration}
A semiorthogonal decomposition of a triangulated category $\D$ is a sequence of full triangulated subcategories $(\A_1,\dots,\A_n)$ in $\D$ such that $\A _i\subset \A_j^{\perp}$ for $1\leq i < j\leq n$ and for every object $X\in \D$ there exists a chain of morphisms in $\D$,

\vspace*{0.2cm}

\[
\xymatrix@C=.5em{
0_{\ } \ar@{=}[r] & X_n \ar[rrrr] &&&& X_{n-1} \ar[rrrr] \ar[dll] &&&& X_{n-2}
\ar[rr] \ar[dll] && \ldots \ar[rr] && X_{1}
\ar[rrrr] &&&& X_0 \ar[dll] \ar@{=}[r] &  X_{\ } \\
&&& A_{n-1} \ar@{->}^{[1]}[ull] &&&& A_{n-2} \ar@{->}^{[1]}[ull] &&&&&&&& A_0\ar@{->}^{[1]}[ull] 
}
\]

\vspace*{0.2cm}

such that a cone $A_k$ of the morphism $X_k\rightarrow X_{k-1}$ belongs to $\A _k$ for $k=1,\dots ,n$.

\vspace*{0.2cm}

\end{definition}


\subsection{Mutations}

Let $\D$ be a triangulated category and assume $\D$ admits a semiorthogonal decomposition 
$\D = \langle \A, \B \rangle$.

\begin{definition}\label{def:left-right_mutation}
The left mutation of $\B$ through $\A$ is defined to be ${\bf L}_{\A}(\B): = \A ^{\perp}$. The 
right mutation of $\A$ through $\B$ is defined to be ${\bf R}_{\B}(\A): = {^{\perp}\B}$.
\end{definition}

One obtains semiorthogonal decompositions $\D = \langle {\bf L}_{\A}(\B), \A\rangle$ and 
$\D = \langle \A , {\bf R}_{\B}(\A)\rangle$.\par 

Let $\A$ be an admissible subcategory of $\D$, and $i : \A\rightarrow \D$ the embedding functor.
It admits a left and a right adjoint functors $\D\rightarrow \A$, the subcategory $\A$ being admissible; denote them $i^{\ast}$ and $i^{!}$, respectively.  Given an object $F \in \D$, define the left mutation ${\bf L}_{\A}(F)$ and the right mutations ${\bf R}_{\A}(F)$  of $F$ through $\A$ by 

\vspace*{0.2cm}

\begin{equation}\label{eq:left-right_mutations_of_an_object}
{\bf L}_{\A}(F): = {\sf Cone}(ii^{!}(F)\rightarrow F), \qquad \qquad {\bf R}_{\A}(F): = {\sf Cone}(F\rightarrow ii^{\ast}(F))[-1].
\end{equation}

\vspace*{0.2cm}

One the proves:

\begin{lemma}\cite[Lemma 2.7]{Kuzrat}\label{lem:mutations_made_explicit}
There are equivalences ${\bf L}_{\A} :\B\simeq \D/\A\simeq {\bf L}_{\A}(\B)$ and ${\bf R}_{\A} :\A\simeq \D/\B\simeq {\bf R}_{\A}(\B)$.
\end{lemma}

\begin{proposition}\cite[Proposition 2.3]{Bo}\label{prop:right-left_mutations_inverse}
Let $\D = \langle \A, \B \rangle$ be as above. Right and left mutations are mutually inverse to each other, i.e. ${\bf R}_{\A}{\bf L}_{\A}\simeq {\rm id}_{\A}$, and ${\bf L}_{\B}{\bf R}_{\B}\simeq {\rm id}_{\B}$.
\end{proposition}

\begin{definition}\label{def:left_and_right_dual_decompositions}
Let $\D = \langle \A _1,\dots,\A_n\rangle$ be a semiorthogonal decomposition. The left dual semiorthogonal decomposition  $\D = \langle \B _n,\dots,\B_1\rangle$ is defined by 

\vspace*{0.2cm}

\begin{equation}
\B _i : = {\bf L}_{\A _1}{\bf L}_{\A _2}\dots {\bf L}_{\A _{i-1}}\A _i = {\bf L}_{\langle \A _1,\dots,\A_{i-1}\rangle}\A _i , \quad 1\leq i \leq n .
\end{equation}

\vspace*{0.2cm}

The right dual semiorthogonal decomposition $\D = \langle \C _n,\dots,\C_1\rangle$ is defined by 

\vspace*{0.2cm}

\begin{equation}
\C _i : = {\bf R}_{\A _n}{\bf R}_{\A _{n-1}}\dots {\bf R}_{\A _{i+1}}\A _i = {\bf R}_{\langle \A _{i+1},\dots,\A_{n}\rangle}\A _i , \quad 1\leq i \leq n .
\end{equation}

\vspace*{0.2cm}

\end{definition}

\begin{lemma}\cite[Lemma 2.10]{Kuzrat}\label{lem:mutations_of_completely_orthogonal_subcategories}
Let $\D = \langle \A _1,\dots,\A_n\rangle$ be a semiorthogonal decomposition such that the components $\A _k$ and $\A _{k+1}$ are completely orthogonal, i.e., $\Hom _{\D}(\A _k,\A _{k+1}) = 0$ and $\Hom _{\D}(\A _{k+1},\A _k) = 0$. Then 

\vspace*{0.2cm}

\begin{equation}
{\bf L}_{\A_k}\A _{k+1} = \A _{k+1} \qquad \qquad and  \qquad \qquad {\bf R}_{\A_{k+1}}\A _k = \A _k,
\end{equation}

\vspace*{0.2cm}

and both the left mutation of $\A _{k+1}$ through ${\A_k}$ and the right mutation of $\A_k$ through $\A _{k+1}$ boil down to a permutation and

\vspace*{0.2cm}

\begin{equation}
\D = \langle \A _1,\dots,\A _{k-1},\A _{k+1},\A _k, \A _{k+2},\dots \A_n\rangle
\end{equation}

\vspace*{0.2cm}

is the resulting semiorthogonal decomposition of $\D$.

\end{lemma}

\subsection{Exceptional collections}

Exceptional collections in $\sf k$--linear triangulated categories are a special case of semiorthogonal decompositions with each component of the decomposition being equivalent to $\Dd ^b({\sf Vect}-{\sf k})$. The above properties of mutations thus specialize to this special case.
Still, there are new features appearing as shown in Subsection \ref{subsec:blockcol}.

\begin{definition}\label{def:exceptcollection}
An object $E \in \D$ of a $\sf k$--linear triangulated category $\D$ is said to be exceptional if there is an isomorphism of graded $\sf k$-algebras 

\vspace{0.2cm}

\begin{equation}
\Hom _{\D}^{\bullet}(E,E) = {\sf k}.
\end{equation}

\vspace{0.2cm}

A collection of exceptional objects $(E_0,\dots,E_n)$ in $\D$ is called 
exceptional if for $1 \leq i < j \leq n$ one has

\vspace{0.2cm}

\begin{equation}
\Hom _{\D}^{\bullet}(E_j,E_i) = 0.
\end{equation}

\vspace{0.2cm}

\end{definition}
Denote $\langle E_0,\dots,E_n \rangle \subset {\D}$ the full
triangulated subcategory generated by the objects $E_0,\dots,E_n$. One 
proves \cite[Theorem 3.2]{Bo} that such a category is admissible. 
The collection $(E_0,\dots,E_n)$ in $\D$ is said to be {\it full} if 
$\langle E_0,\dots,E_n \rangle ^{\perp} = 0$, in other words ${\D}
= \langle E_0,\dots,E_n \rangle$.

 If $\A \subset \D$ is generated by an exceptional object $E$, then by 
 (\ref{eq:left-right_mutations_of_an_object}) the left and right mutations of an object $F\in \D$ through $\A$ are given by the following distinguished triangles:

\vspace*{0.2cm}

\begin{equation}
\RHom _{\D}(E,F)\otimes E\rightarrow F\rightarrow {\bf L}_{\langle E\rangle}(F), \qquad 
{\bf R}_{\langle E\rangle}(F)\rightarrow F\rightarrow \RHom _{\D}(F,E)^{\ast}\otimes E .
\end{equation}

\vspace*{0.2cm}



More generally, if $(E_0,\dots,E_n)$ is an exceptional collection of arbitrary length
in $\D$, then there are  left and right mutations of an object
$E\in {\D}$ through the category $\langle E_0,\dots,E_n
\rangle$, as given in (\ref{eq:left-right_mutations_of_an_object}).
One proves \cite[Proposition 2.1]{Bo} that mutations of an exceptional collection are
exceptional collections.



\subsection{Block collections and block mutations}\label{subsec:blockcol}


\vspace*{0.2cm}

The results of this section are needed for the subsequent Theorem \ref{th:Frobdecomposclaim}.
We follow the exposition of \cite[Section 4]{BrS}.

\begin{definition}\label{def:d-block_collection}
A $d$--block exceptional collection is an exceptional collection $\mathbb E = (E_1,\dots, E_n)$ together with a partition of $\mathbb E$  into $d$ subcollections

\vspace*{0.2cm}

\begin{equation}
\mathbb E = (\mathbb E _1,\dots ,\mathbb E _d),
\end{equation}

\vspace*{0.2cm}

called blocks, such that the objects in each block $\mathbb E _i$ are mutually orthogonal, i.e.
$\Hom ^{\bullet}(E,E') =  0 = \Hom ^{\bullet}(E',E)$ for any $E, E'\in \mathbb E _i$.
\end{definition}

For each integer $1< i\leq d$ we can define an operation $\tau _i$ on $d$--block collections in $\D$ by the rule

\vspace*{0.2cm}

\begin{eqnarray}\label{eq:mutations_block_collections}
& \tau _i (\mathbb E _1,\dots ,\dots \mathbb E _{i-2},\mathbb E _{i-1},\mathbb E _i,\mathbb E _{i+1},\dots \mathbb E _d) = \\
& (\mathbb E _1,\dots ,\dots \mathbb E _{i-2},{\bf L}_{\mathbb E _{i-1}},(\mathbb E _i)[-1],
\mathbb E _{i-1},\mathbb E _{i+1},\dots \mathbb E _d). \nonumber
\end{eqnarray}

\vspace*{0.2cm}

Here, if $\mathbb E _i = (E_{a+1},\dots ,E_b)$ then by definition 

\vspace*{0.2cm}

\begin{equation}
{\bf L}_{\mathbb E _{i-1}}(\mathbb E _i)= ({\bf L}_{\mathbb E _{i-1}}E_{a+1},\dots, {\bf L}_{\mathbb E _{i-1}}E_{b}).
\end{equation}

\vspace*{0.2cm}

\begin{remark}\label{rem:shift_in_mutation}
{\rm Note the shift by $[-1]$ in (\ref{eq:mutations_block_collections}) at ${\bf L}_{\mathbb E _{i-1}},(\mathbb E _i)$; this  will ensure that in the situations below the block mutations $\tau _i$ preserve collections of pure objects.}
\end{remark}

\vspace*{0.2cm}

Recall that a Serre functor (see \cite{BK}) on a $\bf k$--linear triangulated category $\D$ is an autoequivalence ${\mathbb S}_{\D}$ of $\D$ for which there are natural isomorphisms $\Hom _{\D}(E,F) = \Hom _{\D}(F,{\mathbb S}_{\D}(E))^{\ast}$ for $E,F\in \D$. If a Serre functor exists then it is unique up to isomorphism \cite[Proposition 3.4]{BK}.

Given a smooth algebraic variety $X$ over a field $k$, denote $\Dd ^b(X)$ the bounded derived category of coherent sheaves. It is a $k$--linear triangulated category. Let $\omega _X$ be the canonical line bundle on $X$. If $\D = \Dd ^b(X)$ for a smooth projective variety $X$ of dimension $d$, then ${\mathbb S}_{\D} = (-\otimes \omega _X)[d]$.

\begin{theorem}\cite[Theorem 4.5]{BrS}\label{th:mutations_of_block_collections}
Suppose $\mathbb E = (\mathbb E _1,\dots ,\mathbb E _d)$ is a full $d$--block collection and take $1<i\leq d$. Suppose $\D$ is is equipped with a $t$--structure that is preserved by the autoequivalence ${\mathbb S}_{\D}[1-d]$. Then 

\vspace*{0.2cm}

\begin{itemize}

\vspace*{0.2cm}

\item $\mathbb E$ pure implies $\tau _i(\mathbb E)$ pure.

\vspace*{0.2cm}

\item $\mathbb E$ pure implies $\tau _i(\mathbb E)$ strong.

\end{itemize}

\end{theorem}

For our needs, Theorem \ref{th:mutations_of_block_collections} means the following. Let $X$ be a smooth variety of dimension $d-1$, and $\D = \Dd ^b(X)$ equipped with the standard $t$--structure 
$(\Dd ^b(X)^{\leq 0},\Dd ^b(X)^{\geq 0})$. Assume there exists a $d$--block full exceptional collection in $\Dd ^b(X)$ consisting of pure objects, that is, of coherent sheaves in this case.
Then the autoequivalence ${\mathbb S}_{\D}[1-d]$ is just tensoring with $\omega _X$, thus the condition of Theorem \ref{th:mutations_of_block_collections} is immediately satisfied. In this setting, Theorem \ref{th:mutations_of_block_collections} then means that left and right mutations are, too, exceptional collections consisting of coherent sheaves.

\begin{lemma}\cite[Lemma 2.11]{Kuzrat}\label{lem:mutations_canonical_class}
Assume given a semiorthogonal decomposition $\D = \langle \A ,\B\rangle $. Then 

\vspace*{0.2cm}

\begin{equation}
{\bf L}_{\A}(\B)=\B \otimes \omega _X \qquad \qquad and \qquad \qquad  {\bf R}_{\A}(\B)= \A \otimes \omega _X^{-1}.
\end{equation}

\vspace*{0.2cm}

\end{lemma} 


Let $\Ee$ be a vector bundle of rank $r$ on $X$, and consider the associated projective bundle $\pi : \Pp (\Ee)\rightarrow X$. Denote $\Oo _{\pi}(-1)$ the invertible line bundle on $\Pp (\Ee)$ of relative degree $-1$, such that $\pi _{\ast}\Oo _{\pi}(1)=\Ee ^{\ast}$.  One has, \cite{Or}:

\vspace{0.2cm}

\begin{theorem}\label{th:Orvlovth}
The category $\Dd ^b(\Pp (\Ee))$ has a semiorthogonal decomposition: 


\begin{equation}
\Dd ^b(\Pp (\Ee)) = \langle \pi ^{\ast}\Dd ^b(X)\otimes \Oo _{\pi}(-r+1),\dots ,  \pi ^{\ast}\Dd ^b(X)\otimes \Oo _{\pi}(-1),\pi ^{\ast}\Dd ^b(X)\rangle .
\end{equation}

\vspace{0.1cm}

\end{theorem}





\subsection{Dual exceptional collections}


\vspace*{0.2cm}

\begin{definition}\label{def:left-right_dual_exceptional_collections}
Let $X$ be a smooth variety, and assume given an exceptional collection $(E_0,\dots,E_n)$ in $\Dd ^{b}(X)$. The right dual exceptional collection $(F_n,\dots, F_0)$ to $(E_0,\dots,E_n)$ is defined as 

\vspace*{0.2cm}

\begin{equation}
F_i: = {\bf R}_{\langle E_{i+1},\dots,E_{n} \rangle}E_i, \quad {\rm for} \quad 1\leq i\leq n.
\end{equation}

\vspace*{0.2cm}

The left dual exceptional collection $(G_n,\dots, G_0)$ to $(E_0,\dots,E_n)$ is defined as 

\vspace*{0.2cm}

\begin{equation}
G_i: = {\bf L}_{\langle E_{1},\dots,E_{i-1} \rangle}E_i, \quad {\rm for} \quad 1\leq i\leq n.
\end{equation}

\vspace*{0.2cm}

\end{definition}



\begin{proposition}\cite[Proposition 2.15]{Efim}\label{prop:dual_exc_coll_characterization}
Let $(E_0,\dots,E_n)$ be a semiorthogonal decomposition in a triangulated category $\D$. The left dual exceptional collection $\langle F_n, \dots, F_0\rangle$ is uniquely determined by the following property:

\vspace*{0.2cm}

\begin{equation}
\Hom ^l_{\D}(E_i,F _j) = \left\{
   \begin{array}{l}
    {\sf k}, \quad {\rm for} \quad l=0, \  i=j,\\
    0, \quad {\rm otherwise}. \\
   \end{array}
  \right.
\end{equation}

\vspace*{0.2cm}

Similarly, the right dual exceptional collection $\langle G _n,\dots, G _0\rangle $ is uniquely determined by the following property:

\vspace*{0.2cm}

\begin{equation}
\Hom ^l_{\D}(G_i,E_j) = \left\{
   \begin{array}{l}
    {\sf k}, \quad {\rm for} \quad l=0, \  i=j,\\
    0, \quad {\rm otherwise}. \\
   \end{array}
  \right.
\end{equation}

\vspace*{0.2cm}

\end{proposition}

\vspace*{0.5cm}


\section{Resolutions of the diagonal and the decomposition of ${\sf F}_{\ast}\Oo _X$}


\vspace*{0.5cm}
Let $X$ be a smooth variety. Given two (admissible) subcategories $\A$ and $\B$ of $\Dd ^b(X)$, define $\A \boxtimes \B\subset \Dd ^b(X)$ to be the minimal triangulated subcategory of $\Dd ^b(X)$ that contains all the objects $(A\boxtimes B| A\in \A, B\in \B)$. The results of \cite{Kuzbasechange} on base change for semiorthogonal decompositions imply:

\begin{theorem}\label{th:Kuzbasechange}
Let $X$ be as above, and assume given a semiorthogonal decomposition $\langle \A _1,\dots, \A _m\rangle$ of $\Dd ^b(X)$. Let $\langle \C _m,\dots ,\C_1\rangle$ be the right dual semiorthogonal decomposition of $\Dd ^b(X)$ as in Definition \ref{def:left_and_right_dual_decompositions}. 
Then the structure sheaf of the diagonal  $\Delta _{\ast}\Oo _X$ admits (cf. Definition \ref{def:semdecomposition-filtration}) a decomposition $0=D_m\rightarrow D_{m-1}\rightarrow \dots \rightarrow D_1\rightarrow D_0=\Delta _{\ast}\Oo _X$ in $\Dd ^b(X\times X)$, such that ${\rm Cone}(D_i\rightarrow D_{i-1})\in \A _i\boxtimes \C _i^{\vee}\subset \Dd ^b(X\times X)$.
\end{theorem}

\begin{proof}[Sketch of the proof]
The base change theorem for semiorthogonal decompositions \cite[Theorem 5.8]{Kuzbasechange} implies that $\Dd ^b(X\times X)$ has a semiorthogonal decomposition:

\vspace*{0.2cm}

\begin{equation}
\Dd ^b(X\times X) = \langle \A _1\boxtimes \Dd ^b(X),\dots ,\A _m\boxtimes \Dd ^b(X)\rangle ,
\end{equation}

\vspace*{0.2cm}

Consider the decomposition $ D_m\rightarrow D_{m-1}\rightarrow \dots D_1\rightarrow D_0$ of $\Delta _{\ast}\Oo _X$ with respect to the above filtration $\Delta _{\ast}\Oo _X\in \langle \A _1\boxtimes \Dd ^b(X), \dots, \A _m\boxtimes \Dd ^b(X)\rangle$, and let $P_i\in \A _i\boxtimes \Dd ^b(X)$ denotes a cone of the morphism $D_i\rightarrow D _{i-1}$. Proposition \ref{prop:dual_exc_coll_characterization} implies that  $P_i$ belongs, in fact, to $\A _i\boxtimes \C _i^{\vee}\in \A _i\boxtimes \Dd ^b(X)$.
\end{proof}



\begin{corollary}\label{cor:resolution_of_Delta}
Let $X$ be a smooth projective variety, and $(\Ee _0,\dots ,\Ee _m)$ be a full exceptional collection in $\Dd ^b(X)$ with $(\Ff _m, \dots, \Ff _1)$ being its right dual. Then the structure sheaf of the diagonal  $\Delta _{\ast}\Oo _X$ admits  a decomposition $0=D_{m+1}\rightarrow D_m\rightarrow \dots \rightarrow D_1\rightarrow D_0=\Delta _{\ast}\Oo _X$ in $\Dd ^b(X\times X)$, such that a cone of each morphism $D _i\rightarrow D_{i-1}$ is quasiisomorphic to $\Ee _i\boxtimes \Ff  _i^{\vee}$, where $\Ff  _i^{\vee}={\mathcal RHom}(\Ff  _i,\Oo _X)$ is the dual object.

In particular, for any object $\G $ of $\Dd ^b(X)$ there is a spectral sequence 

\vspace*{0.2cm}

\begin{equation}\label{eq:ExcCollSpecSeq}
{\rm E}_1^{p,q}: = {\mathbb H}^{p+q}(X,\G\otimes  \Ff _p^{\vee})\otimes  \Ee _p \Rightarrow \G .
\end{equation}

\vspace*{0.2cm}
\end{corollary}

\begin{proof}
Follows from Theorem \ref{th:Kuzbasechange}. Denoting $p_1,p_2$ the two projections of $X\times X$ onto $X$, for any object $\G \in \Dd ^b(X)$ one has ${p_1}_{\ast}(p_2^{\ast}\G \otimes \Oo _{\Delta _X}) = \G$;  on the other hand, using the decomposition $(D)_{\bullet}$ of $\Oo _{\Delta _X}$ with cones $P_i$ being isomorphic to $\Ee _i\boxtimes \Ff  _i^{\vee}$, one obtains 
a quasiisomorphism $0= D_{m+1}\rightarrow \Ee _m\otimes {\mathbb H}^{\ast}(X,\G \otimes \Ff _m ^{\vee})\rightarrow \dots \rightarrow \Ee _1\otimes {\mathbb H}^{\ast}(X,\G\otimes \Ff _1^{\vee})\rightarrow D_0=\G$. Considering the stupid filtration of the obtained complex, one arrives at (\ref{eq:ExcCollSpecSeq}).
\end{proof}

Assembling together all the previous definitions and statements, we immediately obtain:

\begin{theorem}\label{th:Frobdecomposclaim}
Let $X$ be a smooth variety of dimension $d-1$ over a $\sf k$ of characteristic $p$. Fix an $m\geq 1$ and consider the $m$--th Frobenius morphism ${\sf F}_m$. Assume given a $d$--block (cf. Definition \ref{def:d-block_collection}) full exceptional collection ${\mathbb E} = (\mathbb E _{-d+1},\dots ,\mathbb E _0)$ in $\Dd ^b(X)$ consisting of  coherent sheaves.  Furthermore, assume that for any exceptional object $\Ee \in \mathbb E _i$ and $-d+1\leq i\leq 0$, one has 
${\rm H}^j(X,{\sf F}_m^{\ast}\Ee )=0$ for $j\neq -i$. Denote ${\mathbb G}=(\mathbb G _0,\dots ,\mathbb G _{-d+1})$ the right dual collection. Then 

\vspace*{0.2cm}

\begin{itemize}

\vspace*{0.2cm}

\item[(1)] The right dual collection ${\mathbb G}=(\mathbb G _0,\dots ,\mathbb G _{-d+1})$ is a  $d$--block full exceptional collection.

\vspace*{0.2cm}

\item[(2)] For an exceptional vector bundle $\G \in {\mathbb G}_i$ the corresponding shift is equal to $-i$.

\vspace*{0.2cm}

\item [(3)]There is a decomposition of the bundle ${{\sf F}_n}_{\ast}\Oo _X$ into the direct sum:

\vspace*{0.2cm}

\begin{equation}\label{eq:Decomposition_of_F_*Oo_X}
{{\sf F}_n}_{\ast}\Oo _X = \bigoplus _{i=1}^{i=d} \bigoplus _{\Ee \in \mathbb E _i,  \G \in \mathbb G_{i-d}}{\rm H}^i(X,{\sf F}_n^{\ast}\Ee)\otimes \G^{\vee},
\end{equation}

\vspace*{0.2cm}

and in the inner sum of (\ref{eq:Decomposition_of_F_*Oo_X}) $\G$ is the right dual object for $\Ee$ as in Definition \ref{def:left-right_dual_exceptional_collections}.

\vspace*{0.2cm}

\item [(4)] The terms of ${\mathbb G}$ are, up to a shift, vector bundles on $X$.

\vspace*{0.2cm}

\end{itemize}

\end{theorem}

\begin{proof} 

\begin{itemize}

\vspace{0.2cm}

\item[(1)] Follows from the characterization of right dual collection by Proposition \ref{prop:dual_exc_coll_characterization}.

\vspace{0.2cm}

\item[(2)] Follows from Theorem \ref{th:mutations_of_block_collections}.

\vspace{0.2cm}

\item[(3)] Follows from (2) and from spectral sequence (\ref{eq:ExcCollSpecSeq}) in Corollary \ref{cor:resolution_of_Delta}: namely, the assumption that ${\rm H}^j(X,{\sf F}_m^{\ast}\Ee )=0$ for $j\neq -i$ and for any $\Ee \in \mathbb E _i$ implies that spectral sequence (\ref{eq:ExcCollSpecSeq}) degenerates at the ${\rm E}_1$--term. Indeed, under the theorem's assumptions, Theorem \ref{th:mutations_of_block_collections} ensures that the right dual collection ${\mathbb G}=(\mathbb G _0,\dots ,\mathbb G _{-d+1})$ is pure. Taking into account (\ref{eq:mutations_block_collections}), Definition \ref{def:left_and_right_dual_decompositions}, and Remark \ref{rem:shift_in_mutation}, one sees that if $\G \in \mathbb G _{i-d}$, then by (2) the shifted object $\G [-i]$
is a coherent sheaf. Putting $\G : = {{\sf F}_m}_{\ast}\Oo _X$ in Corollary \ref{cor:resolution_of_Delta} and observing that ${\rm H}^i(X,\Ee \otimes {{\sf F}_m}_{\ast}\Oo _X)={\rm H}^i(X,{\sf F}_m^{\ast}\Ee)$, one obtains a filtration on 
${{\sf F}_m}_{\ast}\Oo _X$ whose associated graded is precisely the right hand side of (\ref{eq:Decomposition_of_F_*Oo_X}). This filtration on ${{\sf F}_m}_{\ast}\Oo _X$ splits, the direct summands of (\ref{eq:Decomposition_of_F_*Oo_X}) being the terms of the right dual exceptional collection $\mathbb G$, hence decomposition (\ref{eq:ExcCollSpecSeq}) follows. 

\vspace{0.2cm}

\item[(4)] Is a consequence of the fact that the terms of the right dual exceptional collection $\mathbb G$ are the direct summands of the coherent sheaf ${{\sf F}_m}_{\ast}\Oo _X$ which is a locally free sheaf, $X$ being smooth.


\end{itemize}

\end{proof}

\vspace*{0.5cm}

\section{Projective spaces and quadrics}\label{sec:P^n_and_quadrics}

\vspace*{0.3cm}

\subsection{Projective spaces} 
To proceed to applications of Theorem \ref{th:Frobdecomposclaim}, we start with the simplest and well--known example. Let $\sf V$ be a vector space of dimension $n+1$ over $\sf k$ of characteristic $p>0$, and $\Pp ^n = \Pp (\sf V)$ be the associated projective space. Beilinson's theorem \cite{Be} asserts that $\Dd ^b(\Pp ^n) = \langle \Oo _{\Pp ^n}(-n),\dots, \Oo _{\Pp ^n}(-1),\Oo _{\Pp ^n}\rangle$. 


\begin{proposition}\label{prop:Beilinson_coll_over_Z}
The collection of line bundles $\langle \Oo _{\Pp ^n}(-n),\dots, \Oo _{\Pp ^n}(-1),\Oo _{\Pp ^n}\rangle $ on $\Pp ^n$ is exceptional.
\end{proposition}

\begin{proof}
This is a consequence of the Serre theorem on cohomology of line bundles on $\Pp ^n_{\mathbb Z}$ (see \cite[III, Theorem 5.1]{Har}).
\end{proof}

\begin{proposition}
The collection from Proposition \ref{prop:Beilinson_coll_over_Z} is full.
\end{proposition}

\begin{proof}
Use the Koszul complexes starting from the Euler sequence to obtain that the category 
$\langle \Oo _{\Pp ^n}(-n),\dots, \Oo _{\Pp ^n}(-1),\Oo _{\Pp ^n}\rangle $
contains all $\Oo (i)$ for $i\in \mathbb Z$ and then conclude by \cite[Example 1.10]{Neem}.
\end{proof}


\begin{lemma}\label{lem:dual_Beilinson_coll_over_Z}
Assuming $p^m>n$, the collection of vector bundles $\langle \Omega _{\Pp ^n}^{n}(n),\dots ,\Omega _{\Pp ^n}^{1}(1),\Oo _{\Pp ^n}\rangle $ satisfies the conditions of Theorem \ref{th:Frobdecomposclaim}.
\end{lemma}

\begin{proof}
From the above one obtains $\Dd ^b(\Pp ^n) = \langle \Oo _{\Pp ^n}(-n),\dots, \Oo _{\Pp ^n}(-1),\Oo _{\Pp ^n}\rangle$; moreover, the right dual to the collection 
$\langle \Oo _{\Pp ^n}(-n),\dots, \Oo _{\Pp ^n}(-1),\Oo _{\Pp ^n}\rangle$ is the collection 
$\langle \Oo _{\Pp ^n},\T _{\Pp ^n}(-1), \dots ,\Oo _{\Pp ^n}(1)\rangle $ (e.g., use the Koszul complexes starting from the Euler sequence to compute it).
By Theorem \ref{th:Frobdecomposclaim} one obtains a resolution of the diagonal 

\vspace{0.3cm}

\begin{equation}
0\rightarrow \Oo _{\Pp ^n}(-1)\boxtimes \Oo _{\Pp ^n}(-n)\rightarrow \dots \rightarrow \Omega _{\Pp ^n}^{1}(1)\boxtimes \Oo _{\Pp ^n}(-1)\rightarrow 
\Oo _{\Pp ^n\times \Pp ^n}\rightarrow \Oo _{\Delta}\rightarrow 0,
\end{equation}

\vspace{0.3cm}

and denote $p_1,p_2$ the two projections of $\Pp ^n\times \Pp ^n$ onto $\Pp ^n$.
Tensoring the resolution along the right $\boxtimes$ factor with $\Oo _{\Pp ^n}(i)$ for $0\leq i\leq n$ and pushing forward the result onto $\Pp ^n$ along $p_1$, one obtains a right Koszul resolution of $\Omega _{\Pp ^n}^{i}(i)$:

\vspace{0.3cm}

\begin{equation}\label{eq:right_Koszul_resolution_for_Omega^i(i)}
0\rightarrow \Omega _{\Pp ^n}^{i}(i)\rightarrow \Omega _{\Pp ^n}^{i-1}(i-1)\otimes {\sf V}^{\ast}\rightarrow  \dots \rightarrow  \Omega _{\Pp ^n}^{1}(1)\otimes {\sf S}^{i-1}{\sf V}^{\ast}\rightarrow {\sf S}^i{\sf V}^{\ast}\otimes \Oo _{\Pp ^n}\rightarrow \Oo _{\Pp ^n}(i)\rightarrow 0.
\end{equation} 

\vspace{0.3cm}

Applying the functor ${\sf F}_m^{\ast}$ to it one can assume, by induction, that ${\rm H}^j(\Pp ^n,{\sf F}_m^{\ast}\Omega _{\Pp ^n}^{k}(k))=0$ for $j>k$ and $0\leq k <i$. Thus,  ${\rm H}^j(\Pp ^n,{\sf F}_m^{\ast}\Omega _{\Pp ^n}^{i}(i))=0$ for $j>i$, the above resolution having the length $i+1$ and the sheaf $\Oo _{\Pp ^n}(i)$ having no higher cohomology by Theorem \ref{th:Kempf}.

On the other hand, tensor the resolution with $\Oo _{\Pp ^n}(-n-1+i)$ along the right $\boxtimes$ -- factor and push forward the result onto $\Pp ^n$ along $p_1$. One obtains now a left Koszul resolution of $\Omega _{\Pp ^n}^{i}(i)$ (rememebering cohomology of line bundles on $\Pp ^n$):

\vspace{0.3cm}

\begin{equation}
0\rightarrow \Oo _{\Pp ^n}(-n-1+i)\rightarrow \Oo _{\Pp ^n}(-1)\otimes {\sf D}^{n-i}({\sf V})\rightarrow	\dots \rightarrow \Omega _{\Pp ^n}^{i+1}(i+1)\otimes {\sf V}\rightarrow \Omega _{\Pp ^n}^{i}(i)\rightarrow 0.
\end{equation}

\vspace{0.3cm}

(${\sf D}$ is the divided power functor). By induction, can assume that ${\rm H}^j(\Pp ^n,{\sf F}_m^{\ast}\Omega _{\Pp ^n}^{k}(k))=0$ for $j<k$ and $i<k <n$. Thus,  ${\rm H}^j(\Pp ^n,{\sf F}_m^{\ast}\Omega _{\Pp ^n}^{i}(i))=0$ for $j<i$, the above resolution having the length $n-i+1$ and the sheaf $\Oo _{\Pp ^n}(-n-1+i)$ having only the top cohomology.

\vspace{0.3cm}

One obtains ${\rm H}^j(\Pp ^n,{\sf F}_m^{\ast}\Omega _{\Pp ^n}^{i}(i))=0$ 
for $j>i$ and ${\rm H}^j(\Pp ^n,{\sf F}_m^{\ast}\Omega _{\Pp ^n}^{i}(i))=0$ for $j<i$. Thus, ${\rm H}^j(\Pp ^n,{\sf F}_m^{\ast}\Omega _{\Pp ^n}^{i}(i))$ can be non-trivial only in the degree $i$; that it is non--trivial (for $p>n$) can be checked computing, for example, the Euler characteristics of the above complexes. On the other hand,  in small characteristics when $p^m<n$ 
the line bundle ${\sf F}_m^{\ast}\Oo _{\Pp ^n}(-1) = \Oo _{\Pp ^n}(-p^m)$ is acyclic; hence the multiplicity space at  $\Oo _{\Pp ^n}(-n)$ is zero which explains the restriction on $p$ in the statement of the lemma.
\end{proof}

\begin{corollary}
One has the decomposition 

\vspace{0.2cm}

\begin{equation}
{{\sf F}_m}_{\ast}\Oo _{\Pp ^n}= \bigoplus _{i=0}^{i=n}\Oo (-i)\otimes {\rm H}^i(\Pp ^n,{\sf F}_m^{\ast}\Omega ^{i}(i)).
\end{equation}

\vspace{0.2cm}

\end{corollary}

\begin{proof}
Follows from Lemma \ref{lem:dual_Beilinson_coll_over_Z} and Theorem \ref{th:Frobdecomposclaim}.
\end{proof}


\subsection{Quadrics} 

A more interesting example is that of smooth quadrics. We assume here that $p$ is odd, thus working over ${\rm Spec}({\mathbb Z}[\frac{1}{2}])$. Consider the semisimple, simply connected group ${\bf Spin}_{n+2}$ and its representation $\nabla _{\omega _1}$ with the highest weight $\omega _1$. Let ${\bf T}\subset {\bf Spin}_{n+2}$ be  a maximal torus. If $n = 2m+1$ or $n = 2m$ then ${\bf Spin}_{n+2}$ has $m + 1$ simple roots $\alpha _1 ,\dots, \alpha _{m+1}$.
A smooth quadric ${\sf Q}_n\subset \Pp (\nabla _{\omega _1})$ is a homogeneous space of ${\bf Spin}_{n+2}$ isomorphic to ${\bf Spin}_{n+2}/{\bf P}_{\alpha _1}$.
We refer to \cite[Section 2]{Lan} for the definition and properties of spinor bundles on smooth quadrics. Depending on the parity of $n$ there are either one spinor bundle denoted by ${\Ss}$ if $n=2k+1$, or two spinor bundles denoted by ${\Ss}_{-},{\Ss}_{+}$ if $n=2k$. The fit into short exact sequences:

\vspace{0.2cm}

\begin{equation}\label{eq:spinor_seq_odd_dim}
0\rightarrow {\Ss}\rightarrow \nabla _{\omega _{k+1}}\otimes \Oo _{{\sf Q}_n}\rightarrow  {\Ss}^{\ast}\rightarrow 0,
\end{equation}

\vspace{0.2cm}

for $n=2k+1$, and 

\vspace{0.2cm}

\begin{equation}\label{eq:spinor_seq_even_dim}
0\rightarrow {\Ss}_{\pm}\rightarrow {\sf U}\otimes \Oo _{{\sf Q}_n}\rightarrow  {\Ss}^{\ast}_{\mp}\rightarrow 0,
\end{equation}

\vspace{0.2cm}

for $n=2k$, where ${\sf U}: = \nabla _{\omega _{k}}$ for the "upper" sequence and ${\sf U}: = \nabla _{\omega _{k+1}}$ for the "lower" sequence (the two spinor representations of ${\bf Spin}_{n+2}$). The positive generator $\Oo _{{\sf Q}_n}(1)$ of ${\rm Pic}({\sf Q}_n)=\mathbb Z$ is isomorphic to the line bundle $\Ll _{\omega _1}$.
Moreover, there are isomorphisms ${\Ss}\otimes \Oo _{{\sf Q}_n}(1)= {\Ss}^{\ast}$ for $n=2k+1$, and ${\Ss}_{\pm}\otimes \Oo _{{\sf Q}_n}(1)={\Ss}^{\ast}_{\mp}$ for $n=2k$.\\ 

For simplicity, we assume below $n=2k+1$, the case of $n=2k$ being similar. The character of the spinor represenation $\nabla _{\omega _k}$ is given by 
by the Weyl character formula (its validity in characteristic $p$ is a consequence of Theorem \ref{th:Kempf}):

\vspace{0.2cm}

\begin{equation}
\chi (\nabla _{\omega _k}) = \frac{(1+e^{\omega _1})(e^{\omega _1}+e^{\omega _2})\dots (e^{\omega _{k-2}}+e^{\omega _{k-1}})(e^{\omega _{k-1}}+e^{2\omega _k})}{e^{\omega _1+\dots +\omega _k}}.
\end{equation}

\vspace{0.2cm}

Using sequence (\ref{eq:spinor_seq_odd_dim}) and the isomorphism ${\Ss}\otimes \Oo _{{\sf Q}_n}(1)= {\Ss}^{\ast}$, one can read off the weights of the spinor bundle $\Ss$. More precisely, denote $\pi: {\bf Spin}_{n+2}/{\bf B}\rightarrow {\bf Spin}_{n+2}/{\bf P}_{\alpha _1}={\sf Q}_n$. 

\begin{example}
{\rm 
Let $n=7$ (i.e., $k=3$). The weights of $\nabla _{\omega _3}$ are equal to 

\vspace{0.2cm}

\begin{equation}
-\omega _3, \ \omega _3-\omega _2, \ -\omega _1+\omega _2-\omega _3, \ \omega _3-\omega _1, \ \omega _1-\omega _3,\ \omega _1-\omega _2+\omega _3, \ \omega _2-\omega _3, \ \omega _3.
\end{equation}

\vspace{0.2cm}

Thus, the bundle $\pi ^{\ast}\Ss$ on the flag variety ${\bf Spin}_{n+2}/{\bf B}$ is filtered by the set of line bundles $\Ll _{-\omega _3}, \ \Ll _{-\omega _1+\omega _3}, \ \Ll_{-\omega _1+\omega _2-\omega _3}, \ \Ll _{-\omega _2+\omega _3}$.
}
\end{example}

\begin{proposition}\label{prop:F^*of_spinor_bundle}
One has ${\rm H}^i({\sf Q}_n,{\sf F}_m^{\ast}\Ss)=0$ for $i\neq \lceil\frac{n}{2}\rceil +1$.
\end{proposition}

\begin{proof}
This follows from \cite[Section 4]{Lan}. For the first--order Frobenius morphism (i.e. for $m=1$) it can also be shown using the defining short exact sequences for the spinor bundles (\ref{eq:spinor_seq_odd_dim}) and (\ref{eq:spinor_seq_even_dim}), and Theorems 	\ref{th:Kempf} and \ref{th:Andcor}. Namely, 
using the knowledge of the weights in the filtration on $\Ss$, one shows, first, that ${\rm H}^i({\sf Q}_n,{\sf F}^{\ast}\Ss)=0$ for $i<\lceil\frac{n}{2}\rceil +1$. Likewise, working out the weights in the dual filtration on $\Ss ^{\ast}$, one obtains ${\rm H}^i({\sf Q}_n,{\sf F}^{\ast}\Ss ^{\ast})=0$ for $i\geq \lceil\frac{n}{2}\rceil +1$.

\end{proof}

\begin{theorem}\label{th:Kap_exc_coll_quadrics}
The collection of vector bundles 

\vspace{0.2cm}

\begin{equation}\label{eq:Kap_exc_coll_quadrics}
{\Ss}(-n), \Oo _{{\sf Q}_n}(-n+1),\dots , \Oo _{{\sf Q}_n}(-1),\Oo _{{\sf Q}_n} 
\end{equation}

\vspace{0.2cm}

on ${\sf Q}_n$ is exceptional and full.
\end{theorem}

\begin{proof}
If the base field of characteristic zero, then this is the content of \cite[Theorem 4.10]{Kap}. In fact, that the collection (\ref{eq:Kap_exc_coll_quadrics}) is exceptional in characteristic $p>n$ can be verified with the help of Proposition \ref{prop:acyclic_weights} and Theorem \ref{th:AndthII}; the latter theorem provides the validity of cohomology calculactions in \cite{Kap} for these primes, the weights of vector bundles in question lying in the interior of the bottom alcove in the dominant chamber in this case. That the collection (\ref
{eq:Kap_exc_coll_quadrics}) is full follows from short exact sequences (\ref{eq:spinor_seq_odd_dim}) and (\ref{eq:spinor_seq_even_dim}), and Lemma \ref{lem:mutations_canonical_class}; namely, from the above sequences one sees that the right mutation of $\Ss (-n)$ through the collection $\langle \Oo _{{\sf Q}_n}(-n+1),\dots , \Oo _{{\sf Q}_n}(-1),\Oo _{{\sf Q}_n}\rangle$ is isomorphic to $\Ss ^{\ast} = \Ss (1) = \Ss (-n)\otimes \omega ^{-1}_{{\sf Q}_n}$. By \cite[Theorem 4.1]{Bo}, the collection (\ref{eq:Kap_exc_coll_quadrics}) is full.
\end{proof}

Thus, sequences (\ref{eq:spinor_seq_odd_dim}) and (\ref{eq:spinor_seq_even_dim}) easily allow to mutate the twisted spinor bundle ${\Ss}(-n)$ to the right within collection (\ref{eq:Kap_exc_coll_quadrics}). 
To obtain a full exceptional collection on ${\sf Q}_n$ that would make Theorem \ref{th:Frobdecomposclaim} applicable, we are going to mutate the bundle ${\Ss}(-n)$ to the right so that it would be situated in the "middle" of sequence (\ref{eq:Kap_exc_coll_quadrics}). In virtue of sequences (\ref{eq:spinor_seq_odd_dim}) and (\ref{eq:spinor_seq_even_dim}), one obtains that the 
right mutation of ${\Ss}(-n)$ is isomorphic to ${\Ss}(-\lceil\frac{n}{2}\rceil)$ (resp., to ${\Ss}_{\pm}(-\lceil\frac{n}{2}\rceil)$). Hence:

\begin{corollary}

The collection of vector bundles 

\vspace{0.2cm}

\begin{equation}\label{eq:mutated_Kap_exc_coll_quadrics}
\Oo _{{\sf Q}_n}(-n+1),\dots , \Oo _{{\sf Q}_n}(-k), \dots ,{\Ss}(-\lceil\frac{n}{2}\rceil), \dots, \Oo _{{\sf Q}_n}(-1),\Oo _{{\sf Q}_n} 
\end{equation}

\vspace{0.2cm}

on ${\sf Q}_n$ is exceptional and full.
\end{corollary}

Let

\vspace{0.2cm}

\begin{equation}\label{eq:dual_Kap_exc_coll_quadrics}
\Oo _{{\sf Q}_n}(-1) \dots \dots , \Phi _k, \dots \dots ,  {\Ss}, \dots \dots ,\Phi _1,\Oo _{{\sf Q}_n} 
\end{equation}

\vspace{0.2cm}

be the right dual collection to (\ref{eq:mutated_Kap_exc_coll_quadrics}). Here some terms are made explicit: the right dual object to $\Oo _{{\sf Q}_n}(-n+1)$ is isomorphic to $(\Oo _{{\sf Q}_n}(-n+1)\otimes \omega ^{-1}_{{\sf Q}_n})^{\ast}=\Oo _{{\sf Q}_n}(-1)$ by Lemma \ref{lem:mutations_canonical_class}; the right dual bundle to ${\Ss}(-\lceil\frac{n}{2}\rceil)$ is isomorphic to $\Ss$
in virtue of sequence (\ref{eq:spinor_seq_odd_dim}). Typically, the right dual to $\Oo _{{\sf Q}_n}(-k)$ is denoted $\Phi _k$.\footnote{Originally, the bundles $\Phi _k$ are denoted by $\Psi _k$ in \cite{Kap}. We reserve the letter $\Psi$ for the subsequent Definition \ref{def:bundles_Psi} that will be used throughout in the sequel, hence the change of notation.}

\begin{lemma}\label{lem:dual_Kapranov_collection_overZ}
Assuming $p^m>n$, the collection (\ref{eq:dual_Kap_exc_coll_quadrics}) satisfies the conditions of Theorem 
\ref{th:Frobdecomposclaim}.
\end{lemma}

\begin{proof}
By Theorem \ref{th:Frobdecomposclaim} one obtains a resolution of the diagonal 

\vspace{0.2cm}

\begin{eqnarray}
& 0\rightarrow  \Oo _{{\sf Q}_n}(-1)\boxtimes\Oo _{{\sf Q}_n}(-n+1)\rightarrow \dots \rightarrow
\Phi _k\boxtimes \Oo _{{\sf Q}_n}(-k)\rightarrow \dots \rightarrow \\
& \dots \rightarrow {\Ss} \boxtimes {\Ss}(-\lceil\frac{n}{2}\rceil)\rightarrow \dots \rightarrow 
\Phi _1\boxtimes \Oo _{{\sf Q}_n}(-1)\rightarrow \Oo _{{\sf Q}_n\times  {\sf Q}_n}\rightarrow \Oo _{\Delta}\rightarrow 0. \nonumber
\end{eqnarray}

\vspace{0.2cm}

Denote $p_1,p_2$ the two projections of ${\sf Q}_n\times  {\sf Q}_n$ onto ${\sf Q}_n$. The argument is similar to that in Lemma \ref{lem:dual_Beilinson_coll_over_Z}.
Tensoring the resolution along the right $\boxtimes$ factor with $\Oo _{{\sf Q}_n}(i)$ for $0\leq i \leq \lceil\frac{n}{2}\rceil$ and pushing forward the result onto ${\sf Q}_n$ along $p_1$, one obtains a right Koszul resolution of $ \Phi _i$:

\vspace{0.2cm}

\begin{equation}
0\rightarrow \Phi _i\rightarrow \Phi _{i-1}\otimes \nabla _{\omega _1}\rightarrow  \dots \rightarrow  \Phi _1\otimes \nabla _{(i-1)\omega _1}\rightarrow \nabla _{i\omega _1}\otimes \Oo _{{\sf Q}_n}\rightarrow \Oo _{{\sf Q}_n}(i)\rightarrow 0.
\end{equation} 

\vspace{0.2cm}

Applying the functor ${\sf F}_m^{\ast}$ to it, one can assume, by induction, that ${\rm H}^j({\sf Q}_n,{\sf F}_m^{\ast}\Phi _k)=0$ for $j>k$ and $0\leq k <i$. Thus,  ${\rm H}^j({\sf Q}_n,{\sf F}_m^{\ast}\Phi _i)=0$ for $j>i$, the above resolution having the length $i+1$ and the sheaf $\Oo _{{\sf Q}_n}(i)$ having no higher cohomology by Theorem \ref{th:Kempf}.

Tensoring  the resolution along the right $\boxtimes$ factor with $\Oo _{{\sf Q}_n}(i)$ for $i>\lceil\frac{n}{2}\rceil$ gives: 

\vspace{0.2cm}

\begin{equation}
0\rightarrow \Phi _i\rightarrow \dots \rightarrow \Phi _{{\lceil\frac{n}{2}\rceil}+1}\otimes \nabla _{\omega _1}\rightarrow \Ss \otimes {\sf W}_i\rightarrow \Phi _{\lceil\frac{n}{2}\rceil}\otimes \nabla _{2\omega _1}\rightarrow  \dots \rightarrow  \Oo _{{\sf Q}_n}(\lceil\frac{n}{2}\rceil+1)\rightarrow 0,
\end{equation} 

\vspace{0.2cm}

where ${\sf W}_i = {\rm H}^0({\sf Q}_n,\Ss (i-\lceil\frac{n}{2}\rceil))$ (this space can be computed using the filtration by line bundles on $\pi ^{\ast}\Ss$). Using Proposition \ref{prop:F^*of_spinor_bundle} and the above inductive argument, one sees that ${\rm H}^j({\sf Q}_n,{\sf F}_m^{\ast}\Phi _i)=0$ for $j>i>\lceil\frac{n}{2}\rceil$ as well.

On the other hand, tensor the resolution with $\Oo _{{\sf Q}_n}(-n-1+i)$ along the right $\boxtimes$ -- factor and push forward the result onto ${\sf Q}_n$ along $p_1$. One obtains now a left Koszul resolution of $\Phi _i$ (rememebering cohomology of line bundles on ${\sf Q}_n$):

\vspace{0.2cm}

\begin{equation}
0\rightarrow \Oo _{{\sf Q}_n}(-n-1+i)\rightarrow \Oo _{{\sf Q}_n}(-1)\otimes \Delta _{(n-i)\omega _1}\rightarrow	\dots \rightarrow \Phi _{i+1}\otimes \Delta _{\omega _1}\rightarrow \Phi _i\rightarrow 0.
\end{equation}

\vspace{0.2cm}

By induction, we can assume that ${\rm H}^j({\sf Q}_n,{\sf F}_m^{\ast}\Phi _k)=0$ for $j<k$ and $i<k <n$. Thus,  ${\rm H}^j({\sf Q}_n,{\sf F}^{\ast}\Phi _i)=0$ for $j<i$, the above resolution having the length $n-i+1$ and the sheaf $\Oo _{{\sf Q}_n}(-n-1+i)$ having only the top cohomology. One obtains ${\rm H}^j({\sf Q}_n,{\sf F}_m^{\ast}\Phi _i)=0$ for $j>i$ and ${\rm H}^j({\sf Q}_n,{\sf F}_m^{\ast}\Phi _i)=0$ for $j<i$. Thus, ${\rm H}^j({\sf Q}_n,{\sf F}_m^{\ast}\Phi _k)$ can be non-trivial only in the degree $k$; that it is non--trivial (for $p^m>n$) can be checked computing, for example, the Euler characteristics of the above complexes.
\end{proof}

\begin{corollary}
One obtains the following decomposition of ${{\sf F}_m}_{\ast}\Oo  _{{\sf Q}_n}$:
\end{corollary}

\vspace{0.2cm}

\begin{equation}
{{\sf F}_m}_{\ast}\Oo  _{{\sf Q}_n}= \bigoplus _{i=0}^{n-1} \Oo _{{\sf Q}_n}(-i)\otimes 
{\rm H}^i({\sf Q}_n,{\sf F}_m^{\ast}\Phi _k)\oplus {\sf S}(-\lceil\frac{n}{2}\rceil)\otimes 
{\rm H}^i({\sf Q}_n,{\sf F}_m^{\ast}{\sf S}).
\end{equation}

\vspace{0.2cm}

\begin{proof}
Follows from Lemma \ref{lem:dual_Kapranov_collection_overZ} and Theorem \ref{th:Frobdecomposclaim}.
\end{proof}

\vspace{0.5cm}

\section{Flag varieties of semi-simple rank two groups}\label{sec:Frobeniusdecompositions}

\vspace{0.5cm}

In this section we find semiorthogonal decompositions of the derived categories of flag varieties 
of semi-simple rank two groups, i.e. of the types ${\bf A}_2$ and ${\bf B}_2$. 
We explicitly compute right dual collections for all the above groups that allow to obtain the decomposition of ${{\sf F}_n}_{\ast}\Oo _{{\bf G}/{\bf B}}$ for $\bf G$ as above.


\subsection{Partial decompositions}


We start with a simple observation about semiorthogonal decompositions that partly satisfy the conditions of Theorem \ref{th:Frobdecomposclaim} and hold for arbitrary groups.

\begin{definition}\label{def:bundles_Psi}
Given a semisimple algebraic group ${\bf G}$ and a fundamental weight $\omega$ of $\bf G$, for $1\leq i\leq l$, where $l={\rm dim} (\nabla _{\omega _k})$ the bundle $\Psi ^{\omega _k}_i$ is set to be the pull--back of $\Omega _{\Pp (\nabla _{\omega _k})}^i(i)$ along the morphism ${\bf G}/{\bf B}\rightarrow \Pp (\nabla _{\omega _k})$ defined by a semi--ample line bundle $\Ll _{\omega _k}$. The latter morphism factors through the projection $\pi _k: {\bf G}/{\bf B}\rightarrow {\bf G}/{\bf P}_{\hat \alpha _k}$ and the embedding ${\bf G}/{\bf P}_{\hat \alpha}\hookrightarrow \Pp (\nabla _{\omega _k})$, where ${\bf P}_{\hat \alpha _k}$ is the maximal parabolic subgroup of $\bf G$ obtained by adjoining all but one the simple roots to $\bf B$; the deleted simple root is the one with $\langle \omega _k , \alpha _k^{\vee}\rangle =1$.
\end{definition}

By definition, the bundles $\Psi _1^{\omega _k}$ fit into short exact sequences:

\vspace*{0.2cm}

\begin{equation}\label{eq:defining_sequence_for_Psi}
0\rightarrow \Psi _1^{\omega _k}\rightarrow \nabla _{\omega _k}\otimes \Oo _{{\bf G}/{\bf B}}\rightarrow 
\Ll _{\omega _k}\rightarrow 0,
\end{equation}

\vspace*{0.2cm}

More generally,  for $1\leq i\leq l$ one has the Koszul complexes (cf. (\ref{eq:right_Koszul_resolution_for_Omega^i(i)})):

\vspace*{0.2cm}

\begin{equation}\label{eq:defining_sequence_for_Psi_i}
0\rightarrow \Psi _i^{\omega _k}\rightarrow \Psi _{i-1}^{\omega _k}\otimes \nabla _{\omega _k}\rightarrow  \dots \rightarrow   \Psi _1^{\omega _k}\otimes {\sf S}^{i-1}\nabla _{\omega _k}\rightarrow {\sf S}^i\nabla _{\omega _k}\otimes \Oo _{{\bf G}/{\bf B}}\rightarrow \Ll _{i\omega _k}\rightarrow 0.
\end{equation}

\vspace*{0.2cm}

\begin{proposition}\label{prop:simple_partial_decomp}
Let $\bf G$ be a group of rank $n$, and $\omega _1,\dots ,\omega _n$ the fundamental weights.
Consider a sequence $\A = \langle \A _i\rangle _{i=-4}^{i=0}$ of full triangulated subcategories of $\Dd ^b({\bf G}/{\bf B})$:

\vspace*{0.2cm}

\begin{eqnarray}\label{eq:partial_decomposition}
& \A _{-3}=\langle \Ll _{-\rho}\rangle, \quad \A _{-2}, \quad \A _{-1}=\langle \Psi _1^{\omega _1},\dots, \Psi _1^{\omega _n} \rangle, \quad \A _0=\langle \Oo _{{\bf SL}_3/{\bf B}}\rangle , 
\end{eqnarray}

\vspace*{0.2cm}

where $\A _{-2}: = {^{\perp}\langle \A _{-3}\rangle} \cap {\langle \A _{-1}, \A _{0}\rangle}^{\perp}$. Then the components $\A _{-3}, \A _{-1},\A _{0}$ satisfy the conditions of Theorem \ref{th:Frobdecomposclaim}.
\end{proposition}

\begin{proof}
Semiorthogonality of sequence (\ref{eq:partial_decomposition}) is immediately verified using short exact sequences defining the bundles $\Psi _1^{\omega _k}$

\vspace*{0.2cm}

\begin{equation}\label{eq:defining_sequence_for_Psi_1}
0\rightarrow \Psi _1^{\omega _k}\rightarrow \nabla _{\omega _k}\otimes \Oo _{{\bf G}/{\bf B}}\rightarrow 
\Ll _{\omega _k}\rightarrow 0,
\end{equation}

\vspace*{0.2cm}

and Theorem \ref{th:Bott-Demazure_th}. Theorem \ref{th:Kempf} gives ${\rm H}^i({\bf G}/{\bf B},\Oo _{{\bf G}/{\bf B}})=0$ for $i>0$, which proves the statement for the category $\A _0$. Putting $d={\rm dim}({\bf G}/{\bf B})$, Serre's duality on ${\bf G}/{\bf B}$ implies:

\vspace*{0.2cm}

\begin{equation}\label{eq:Serre_duality_for_-rho}
{\rm H}^i({\bf G}/{\bf B},\Ll _{-p^n\rho}) = {\rm H}^{d-i}({\bf G}/{\bf B},\Ll _{(p^n-2)\rho})^{\ast}.
\end{equation}

\vspace*{0.2cm}

If $p=2$ and $n=1$, then the right hand side of (\ref{eq:Serre_duality_for_-rho}) is isomorphic to ${\rm H}^{d-i}({\bf G}/{\bf B},\Oo _{{\bf G}/{\bf B}})^{\ast}$. If either $p>2$ or $n>1$, then line bundle $\Ll _{(p^n-2)\rho}$ is ample on ${\bf G}/{\bf B}$. In both cases the group 
${\rm H}^{d-i}({\bf G}/{\bf B},\Ll _{(p^n-2)\rho})$ vanishes for $d-i>0$, hence ${\rm H}^i({\bf G}/{\bf B},\Ll _{-p^n\rho})$ is non--zero only for $i=d$. 

Applying ${\sf F}_n^{\ast}$ to (\ref{eq:defining_sequence_for_Psi}) and passing to the cohomology, one obtains 

\vspace*{0.2cm}

\begin{eqnarray}
& 0\rightarrow {\rm H}^0({\bf G}/{\bf B},{\sf F}_n^{\ast}\Psi _1^{\omega _k})\rightarrow 
\nabla _{\omega _k}^{[n]}\otimes {\rm H}^0({\bf G}/{\bf B},\Oo _{{\bf G}/{\bf B}})\rightarrow {\rm H}^0({\bf G}/{\bf B},\Ll _{p^n\omega _k})  \\
& {\rm H}^1({\bf G}/{\bf B},{\sf F}_n^{\ast}\Psi _1^{\omega _k})\rightarrow 
\nabla _{\omega _k}^{[n]}\otimes {\rm H}^1({\bf G}/{\bf B},\Oo _{{\bf G}/{\bf B}})\rightarrow 
{\rm H}^1({\bf G}/{\bf B},\Ll _{p^n\omega _k}) \rightarrow  \dots
\dots  \nonumber
\end{eqnarray}

\vspace*{0.2cm}

The higher cohomology groups ${\rm H}^i({\bf G}/{\bf B},\Oo _{{\bf G}/{\bf B}})$ and 
${\rm H}^1({\bf G}/{\bf B},\Ll _{p^n\omega _k})$ vanish for $i>0$ by Theorem \ref{th:Kempf}, and the map $\nabla _{\omega _k}^{[n]}\otimes {\rm H}^0({\bf G}/{\bf B},\Oo _{{\bf G}/{\bf B}})=\nabla _{\omega _k}^{[n]}\rightarrow {\rm H}^0({\bf G}/{\bf B},\Ll _{p^n\omega _k})=\nabla _{p^n\omega _k}$ is the embedding of the Frobenius twist of an irreducible\footnote{
Away from the very small primes for which the induced module $\nabla _{\omega _k}$ is reducible.}
 representation $\nabla _{\omega _k}$ into $\nabla _{p^n\omega _k}$, hence the statement.
\end{proof}

\vspace*{0.2cm}


Denote $d_{\alpha}$ the dimension ${\bf G}/{\bf P}_{\hat \alpha}$, and let $k_{\alpha}$ be the index of ${\bf G}/{\bf P}_{\hat \alpha}$ which is defined from the equation $\omega _{{\bf G}/{\bf P}_{\hat \alpha}} = \Ll _{-k_{\alpha}\omega}$ (recall that $\Ll _{\omega}$ is the positive generator of the group ${\rm Pic}({\bf G}/{\bf P}_{\hat \alpha}) = \mathbb Z$).
Similarly to (\ref{eq:Serre_duality_for_-rho}), one obtains ${\rm H}^i({\bf G}/{\bf P}_{\hat \alpha},\Ll _{-p^n\omega}) = {\rm H}^{d_{\alpha}-i}({\bf G}/{\bf B},\Ll _{(p^n-k_{\alpha})\omega})^{\ast}$. Hence, ${\rm H}^i({\bf G}/{\bf P}_{\hat \alpha},\Ll _{-p^n\omega})=0$ for $i\neq d_{\alpha}$. 
When $\bf G$ has a small rank, decomposition (\ref{eq:partial_decomposition}) can be refined; that is, one can split another semiorthogonal piece off $\A _{-2}$. Specifically, consider a sequence of triangulated subcategories:

\begin{figure}[H]
\begin{equation}\label{eq:FrobExcCollonSl_3/B}
\xymatrix{
& \A _{-4} & \A _{-3}&  \A _{-2}& \A _{-2}& \A _0\\
& || & || & || & ||  & || \\
&*++<10pt>[F]\txt{$\Ll _{-\rho}$}
&*++<10pt>[F]\txt{$\Ll _{-\omega _1}$ \\ \\ \vdots \\ \\ $\Ll _{-\omega _n}$}
&*++<10pt>[F]\txt{${^{\perp}\langle \A _{-4}, \A _{-3}\rangle} \cap {\langle \A _{-1}, \A _{0}\rangle}^{\perp}$}
&*++<10pt>[F]\txt{$\Psi _1^{\omega _1}$  \\ \\ \vdots \\ \\ $\Psi _1^{\omega _n}$}
&*++<10pt>[F]\txt{$\Oo _{{\bf G}/{\bf B}}$}
}
\end{equation}
 \end{figure}


\begin{proposition}\label{prop:enhanced_simple_partial_decomp}
Let $\bf G$ be of rank two. Then the components $\A _{-4}, \A _{-3}, A _{-1},\A _{0}$ satisfy the conditions of Theorem \ref{th:Frobdecomposclaim}.
\end{proposition}




\subsection{Type ${\bf A}_2$}


Let $\bf G = {\bf SL}_3$, and $\omega _1,\omega _2$ the two fundamental weights. 
Recall that the canonical sheaf $\omega _{{\bf G}/{\bf B}}$ on ${\bf G}/{\bf B}$ is isomorphic to $\Ll _{-2\rho}$, where $\rho$ is the sum of fundamental weights. 

\begin{lemma}\label{lem:SL_3-lemma}
Consider a sequence $\A = \langle \A _i\rangle _{i=-3}^{i=0}$ of full triangulated subcategories of $\Dd ^b({\bf G}/{\bf B})$:

\vspace*{0.2cm}

\begin{figure}[H]
\begin{equation}\label{eq:FrobExcCollonSl_3/B}
\xymatrix{
& \A _{-3}& \A _{-2}& \A _{-1}& \A _0\\
& || & || & || & || \\
&*++<10pt>[F]\txt{$\Ll _{-\rho}$}
&*++<10pt>[F]\txt{$\Ll _{-\omega _1}$  \\ \\ $\Ll _{-\omega _2}$}
&*++<10pt>[F]\txt{$\Psi _1^{\omega _1}$  \\ \\ $\Psi _1^{\omega _2}  $}
 &*++<10pt>[F]\txt{$\Oo _{{\bf SL}_3/{\bf B}}$}
}
\end{equation}
 \end{figure}
 
 \vspace*{0.2cm}

Then $ \A$ is a semiorthogonal decomposition of $\Dd ^b({\bf SL}_3/{\bf B})$ satisfying the conditions of Theorem \ref{th:Frobdecomposclaim}. The right dual decomposition with respect to (\ref{eq:FrobExcCollonSl_3/B}) consists of the following subcategories:

\vspace*{0.2cm}

\begin{figure}[H]
\begin{equation}\label{eq:FrobleftdualExcCollonSl_3/B}
\xymatrix{
& \C _0& \C _1& \C _2 & \C _3 \\
& || & || & || & || \\
&*++<10pt>[F]\txt{$\Oo _{{\bf SL}_3/{\bf B}}$}
&*++<10pt>[F]\txt{$\Ll _{\omega _1}$  \\ \\ $\Ll _{\omega _2}$}
&*++<10pt>[F]\txt{$\Psi _1^{\omega _1}\otimes \Ll _{\rho}$  \\ \\ $\Psi _1^{\omega _2}\otimes \Ll _{\rho}$}
&*++<10pt>[F]\txt{$\Ll _{\rho}$}
}
\end{equation}
 \end{figure}

\vspace*{0.2cm}

The bundles in the subcategory $\C _i$ for $-3 \leq i\leq 0$ are shifted by $[i]$ .

\end{lemma}

\begin{proof}     
Taking into account Propositions \ref{prop:simple_partial_decomp} and \ref{prop:enhanced_simple_partial_decomp}, one sees that the statement of the lemma is equivalent to saying that in the given case the subcategory $\A _{-2}$ 
in sequence (\ref{eq:partial_decomposition}) from Proposition \ref{prop:simple_partial_decomp} is zero.
To show this, consider short exact sequences:

\vspace*{0.2cm}

\begin{equation}
0\rightarrow \Ll _{-\omega _2}\rightarrow \Psi _1^{\omega _1}\rightarrow  \Ll _{\omega _2-\omega _1}\rightarrow 0,
\end{equation}

\vspace*{0.2cm}

\begin{equation}\label{eq:2ndrelEulerseq}
0\rightarrow \Ll _{-\omega _1}\rightarrow \Psi _1^{\omega _2}\rightarrow  \Ll _{\omega _1-\omega _2}\rightarrow 0.
\end{equation}

\vspace*{0.2cm}



It follows, for example, from sequence (\ref{eq:2ndrelEulerseq}) that the bundle $\Ll _{\omega _1-\omega _2}$ belongs to $\A$. Given that $\Ll _{-\omega _2},\Ll _{-\rho}$ also belong to $\A$, it follows from Theorem \ref{th:Orvlovth} that $\A$ generates the whole 
$\Dd ^b({\bf SL}_3/{\bf B})$.

The terms of right dual decomposition can easily be found using sequences (\ref{eq:defining_sequence_for_Psi_1}) for $k=1,2$. Thus, one obtains the right dual 
to $\Psi _1^{\omega _1}$ being isomorphic to $\Ll _{\omega _1}[-1]$, while the right dual to $\Psi _1^{\omega _2}$ is found to be isomorphic to $\Ll _{\omega _2}[-1]$. On the other hand, to find the 
right duals to $\Ll _{-\omega _1}, \Ll _{-\omega _2}$, one can mutate these to the left through $\Ll _{-\rho}$, and then to the right 
through the whole collection. Proposition \ref{prop:right-left_mutations_inverse} ensures that these will amount to the same result. The effect of the last action is described by Theorem \ref{th:mutations_of_block_collections}, that is tensoring with $\Ll _{2\rho}$. Left mutations through $\Ll _{-\rho}$ can be found tensoring sequences (\ref{eq:defining_sequence_for_Psi_1}) with $\Ll _{-\rho}$ or $k=1,2$. One sees that the right duals are isomorphic to $\Psi _1^{\omega _1}\otimes \Ll _{\rho}[-2]$, and $\Psi _1^{\omega _2}\otimes \Ll _{\rho}[-2]$. Finally, the right dual to $\Ll _{-\rho}$ is isomorphic to $\Ll _{\rho}[-3]$, once again by Theorem \ref{th:mutations_of_block_collections}.
\end{proof}

\vspace*{0.2cm}

\begin{theorem}\label{th:FrobdecomSL_3/B}
The bundle ${\sf F _n}_{\ast}\Oo _{{\bf SL}_3/{\bf B}}$ decomposes into the direct sum of vector bundles with indecomposable summands being isomorphic to:

\vspace*{0.2cm}

\begin{equation}\label{eq:FrobrightdualcollSL_3_B}
\Oo _{{\bf SL}_3/{\bf B}}, \quad  \Ll_{-\omega _1},   \quad \Ll _{-\omega _2},   \quad (\Psi _1^{\omega _1})^{\ast}\otimes \Ll _{-\rho},  \quad (\Psi _1^{\omega _2})^{\ast}\otimes \Ll _{-\rho},  \quad \Ll _{-\rho}.
\end{equation}

\vspace*{0.2cm}

The multiplicity spaces at each indecomposable summand are isomorphic, respectively, to:

\vspace*{0.2cm}

\begin{equation}
{\sf k}, \quad  \nabla _{p^n\omega _2}/\nabla _{\omega _2}^{[n]}, \quad \nabla _{p^n\omega _1}/\nabla _{\omega _1}^{[n]},
\quad \Delta _{(p^n-3)\omega _1}, \quad \Delta _{(p^n-3)\omega _2}, \quad \Delta _{(p^n-2)\rho}.
\end{equation}

\vspace*{0.2cm}

The set of bundles in (\ref{eq:FrobrightdualcollSL_3_B}) in the decomposition of 
${\sf F _n}_{\ast}\Oo _{{\bf SL}_3/{\bf B}}$ forms a full exceptional collection in $\Dd ^b({\bf SL}_3/{\bf B})$.

\end{theorem}

\begin{proof}
The semiorthogonal decomposition $\A$ from Lemma \ref{lem:SL_3-lemma} satisfies the assumptions of Theorem \ref{th:Frobdecomposclaim}. Hence, one has:

\vspace*{0.2cm}

\begin{equation}
{\sf F _n}_{\ast}\Oo _{{\bf SL}_3/{\bf B}}=  \bigoplus {\mathbb H}^{p}(X,{\sf F}_n^{\ast}\Ee _p)\otimes \Ff _p^{\ast},
\end{equation}

\vspace*{0.2cm}

where $\Ee _p$ are the terms of $\A$. Recall that $\Ff _p^{\ast}$ are obtained by dualizing the terms of right dual decomposition that are precisely the set in (\ref{eq:FrobrightdualcollSL_3_B}).
\end{proof}

\begin{remark}
{\rm The decomposition of ${{\sf F}_n}_{\ast}\Oo _{{\bf SL}_3/{\bf B}}$ for $n=1$ was previously obtained in \cite{HKR} using results of 
\cite{AK89} about the structure of ${\bf G}_1{\bf T}$--socle series of the induced ${\bf G}_1{\bf B}$--module $\hat{\nabla}(0)={\rm Ind}_{\bf B}^{{\bf G}_1{\bf B}}\chi _{0}$. The method of {\it loc.cit.} is drastically different from ours.}
\end{remark}


\subsection{Type ${\bf B}_2$}


We assume here that $p$ is odd, thus working over ${\rm Spec}({\mathbb Z}[\frac{1}{2}])$.
Recall the necessary facts about the flag variety ${\bf Sp}_4/{\bf B}$.
The group ${\bf Sp}_4$ has two parabolic subgroups ${\bf P}_{\alpha}$ and ${\bf
  P}_{\beta}$ that correspond to the simple roots $\alpha$ and
$\beta$, the root $\beta$ being the long root. Specifically, it follows from the description of parabolic subgroups of classical groups (e.g., \cite[Proposition 12.13]{MalTest}) that the homogeneous spaces ${\bf G}/{{\bf P}_{\alpha}}$ and  ${\bf G}/{{\bf P}_{\beta}}$ are isomorphic to the 3-dimensional quadric ${\sf Q}_3$ and $\Pp ^3$, respectively. The Levi subgroups of ${\bf P}_{\alpha}$ and ${\bf
  P}_{\beta}$ have a component isomorphic to ${\bf SL}_2$; its tautological representation in each case gives rise to a homogeneous rank 2 vector bundle on ${\bf G}/{{\bf P}_{\alpha}}$ (resp., on ${\bf G}/{{\bf P}_{\beta}}$). Denote 
$\pi _{\alpha}$ and $\pi _{\beta}$ the two projections of ${\bf G}/{\bf B}$ onto $\Pp ^3$ and ${\sf Q}_3$, respectively. More concretely, the projection $\pi _{\alpha}$ is the projective bundle over $\Pp ^3$ associated to a rank two vector bundle $\sf N$ over $\Pp ^3=\Pp (\nabla _{\omega _{\alpha}})$, and the projection $\pi _{\beta}$ is the projective bundle associated to another rank two bundle $\Uu _2$ on ${\sf Q}_3\subset \Pp (\nabla _{\omega _{\beta}})$. There is a short exact sequence on $\Pp ^3$:

	
\begin{equation}\label{eq:null_corr_ses}
0\rightarrow \Ll _{-\omega _{\alpha}}\rightarrow \Psi _1^{\omega _{\alpha}}\rightarrow {\sf
  N}\rightarrow 0,
\end{equation}


while the bundle $\Uu _2$ fits into a short exact sequence on ${\sf Q}_3$:


\begin{equation}\label{eq:spinor_ses}
0\rightarrow \Uu _2\rightarrow {\nabla}_{\omega _{\alpha}}\otimes \Oo _{{\sf
    Q}_3}\rightarrow \Uu _2^{\ast}\rightarrow 0.
\end{equation}

 
The following short exact sequences on ${\bf Sp}_4/{\bf B}$ will also be useful:


\begin{equation}\label{eq:seqforNonSp_4/B}
0\rightarrow \Ll _{\omega _{\alpha}-\omega _{\beta}}\rightarrow \pi _{\alpha}^{\ast}{\sf N}\rightarrow \Ll _{\omega _{\beta}-\omega _{\alpha}}\rightarrow 0,
\end{equation}


and 


\begin{equation}\label{eq:U2_seq}
0\rightarrow \Ll _{-\omega _{\alpha}}\rightarrow \pi _{\beta}^{\ast}\Uu _2\rightarrow \Ll _{\omega _{\alpha}-\omega _{\beta}}\rightarrow 0.
\end{equation}


Consider the sequence $\A= \langle \A _i\rangle _{i=-4}^{i=0}$ of full triangulated subcategories of $\Dd ^b({\bf Sp}_4/{\bf B})$:


\begin{figure}[H]
\begin{equation}\label{eq:FrobExcCollonSp_4/B}
$$
\xymatrix{
\A _{-4} & \A _{-3}& \A _{-2}& \A _{-1}& \A _0\\
|| & || & || & || & || \\
*++<10pt>[F]\txt{$\Ll _{-\rho}$}
&*++<10pt>[F]\txt{$\Ll _{-\omega _{\alpha}}$  \\ \\ $\Ll _{-\omega _{\beta}}$}
&*++<10pt>[F]\txt{$\Psi _2^{\omega _{\alpha}}$  \\ \\ $\pi _{\beta}^{\ast}\Uu _2$}
&*++<10pt>[F]\txt{$\Psi _1^{\omega _{\alpha}}$  \\ \\ $\Psi _1^{\omega _{\beta}}  $}
 &*++<10pt>[F]\txt{$\Oo _{{\bf Sp}_4/{\bf B}}$}
}
$$
\end{equation}
 \end{figure}

\vspace*{0.2cm}

\begin{lemma}\label{lem:Sp_4-lemma}
The sequence $\A$ is a semiorthogonal decomposition of $\Dd ^b({\bf Sp}_4/{\bf B})$ satisfying the conditions of Theorem \ref{th:Frobdecomposclaim}. The right dual decomposition $\C$ with respect to (\ref{eq:FrobExcCollonSp_4/B}) consists of the following subcategories:

\vspace*{0.2cm}

\vspace*{0.2cm}

\begin{figure}[H]
$$
\xymatrix{
\C _{0}& \C _{1}& \C _{2}& \C _{3 }& \C _{4}\\
|| & || & || & || & || \\
*++<10pt>[F]\txt{$\Oo _{{\bf Sp}_4/{\bf B}}$}
&*++<10pt>[F]\txt{$\Ll _{\omega _{\alpha}}$ \\ \\ $\Ll _{\omega _{\beta}}$}
&*++<10pt>[F]\txt{$\pi _{\beta}^{\ast}\Uu _2\otimes \Ll _{\rho}$ \\ \\ $\Psi _2^{\omega _{\alpha}}\otimes \Ll _{\rho}$}
&*++<10pt>[F]\txt{$\Psi _1^{\omega _{\beta}}\otimes \Ll _{\rho}$ \\ \\ $\Psi _1^{\omega _{\alpha}}$}
 &*++<10pt>[F]\txt{$\Ll _{\rho}$}
}
$$
 \end{figure}

\vspace*{0.2cm}

The bundles in the subcategory $\C _i$ for $-4 \leq i\leq 0$ are shifted by $[i]$ .

\end{lemma}

\begin{proof}
The collections $\langle \Ll _{-\omega _{\alpha}},\Psi _2^{\omega _{\alpha}},\Psi _1^{\omega _{\alpha}},\Oo _{\Pp ^3}\rangle$ and $\langle \Ll _{-\omega _{\beta}},\Uu _2,\Psi _1^{\omega _{\beta}},\Oo _{{\sf Q}_3}\rangle$ are exceptional on $\Pp ^3$  (see \cite{Be}) and ${\sf Q}_3$ (see \cite{Kap}), respectively. This implies many orthogonalities in the collection of the lemma. That the bundles in each $\A _i$ are mutually orthogonal can be verified using short exact sequences:

\vspace*{0.2cm}

\begin{equation}\label{eq:EuleronP^3forB_2}
0\rightarrow \Ll _{-\omega _{\alpha}}\rightarrow {\nabla}_{\omega _{\alpha}}\otimes \Oo _{\Pp ^3}\rightarrow (\Psi _1^{\omega _{\alpha}})^{\ast}\rightarrow 0,
\end{equation}

\vspace*{0.2cm}

\begin{equation}\label{eq:EuleronQ_3forB_2}
0\rightarrow \Ll _{-\omega _{\beta}}\rightarrow (\nabla _{\omega _{\beta}})^{\ast}\otimes \Oo _{{\sf Q}_3} \rightarrow (\Psi _1^{\omega _{\beta}})^{\ast}\rightarrow 0,
\end{equation}	

\vspace*{0.2cm}

and (\ref{eq:U2_seq}). Using these sequences and Theorems \ref{th:Bott-Demazure_th} and \ref{th:Andcor}, one deduces as well the remaining set of orthogonalities.

The property ${\rm H}^j({\bf Sp}_4/{\bf B},{\sf F}_n^{\ast}\Ee)=0$ for $j\neq i$ for a bundle $\Ee$ belonging to $\A_{-i}$ can be obtained, for example, from decompositions ${{\sf F}_n}_{\ast}\Oo _{\Pp ^3}$ and ${{\sf F}_n}_{\ast}\Oo _{{\sf Q}_3}$ into direct sum of indecomposables:

\vspace*{0.2cm}

\begin{equation}
{{\sf F}_n}_{\ast}\Oo _{\Pp ^3} = \Oo _{\Pp ^3}\oplus \Ll _{-\omega _{\alpha}}^{\oplus p_1}\oplus \Ll _{-2\omega _{\alpha}}^{\oplus p_2}
\oplus \Ll _{-3\omega _{\alpha}}^{\oplus p_3}
\end{equation}

\vspace*{0.2cm}

\begin{equation}
{{\sf F}_n}_{\ast}\Oo _{{\sf Q}_3}=  \Oo _{{\sf Q}_3}\oplus \Ll _{-\omega _{\beta}}^{\oplus q_1}\oplus (\Uu _2\otimes \Ll _{\omega _{-\beta}})^{\oplus q_2}\oplus \Ll _{-2\omega _{\beta}}^{\oplus q_3}.
\end{equation}

\vspace*{0.2cm}

Finally, that $\A$ generates $\Dd ^b({\bf Sp}_4/{\bf B})$
easily follows from sequences (\ref{eq:seqforNonSp_4/B}), (\ref{eq:U2_seq}), and (\ref{eq:defining_sequence_for_Psi_1}) with $\omega _k=\omega _{\beta}$.
Using these, one sees that semiorthogonal sequence (\ref{eq:FrobExcCollonSp_4/B}) contains the subcategories $\pi _{\beta}^{\ast}\Dd ^b({\sf Q}_3)$ and $\pi _{\beta}^{\ast}\Dd ^b({\sf Q}_3)\otimes \Ll _{-\omega _{\alpha}}$, hence delivering a semiorthogonal decomposition by Theorem \ref{th:Orvlovth}.

\subsubsection{Computing the right dual collection}
From sequences dual to (\ref{eq:EuleronP^3forB_2}) and (\ref{eq:defining_sequence_for_Psi_1}) with $\omega _k=\omega _{\beta}$, one immediately sees that ${\bf R}_{\langle \Oo _{{\bf Sp}_4/{\bf B}}\rangle}(\Psi _1^{\omega _{\alpha}})=\Ll _{\omega _{\alpha}}[-1]$ and  ${\bf R}_{\langle \Oo _{{\bf Sp}_4/{\bf B}}\rangle}(\Psi _1^{\omega _{\beta}})=\Ll _{\omega _{\beta}}[-1]$. On the other hand, tensoring each of the above sequences with $\Ll _{-\rho}$, 
one has from (\ref{eq:EuleronP^3forB_2}) that ${\bf L}_{\langle \Ll _{-\rho}\rangle}(\Ll _{-\omega _{\beta}})=\Psi _1^{\omega _{\alpha}}\otimes \Ll _{-\rho}[1]$ and from (\ref{eq:defining_sequence_for_Psi_1}) with $\omega _k=\omega _{\beta}$ that ${\bf L}_{\langle \Ll _{-\rho}\rangle}(\Ll _{-\omega _{\alpha}})=\Psi _1^{\omega _{\beta}}\otimes \Ll _{-\rho}[1]$. Thus, by 
 Proposition \ref{prop:right-left_mutations_inverse}, Theorem \ref{th:mutations_of_block_collections}, and Lemma \ref{lem:mutations_canonical_class},
 one obtains:

\vspace*{0.2cm}

\begin{equation}
{\bf R}_{\langle {\A}\rangle}{\bf L}_{\A _{-4}}(\Ll _{-\omega _{\alpha}}) = 
{\bf R}_{\langle {\A}_{-2},{\A}_{-1},{\A}_0\rangle}(\Ll _{-\omega _{\alpha}}) = \Psi _1^{\omega _{\beta}}\otimes \Ll _{\rho}[-3],
\end{equation}

\vspace*{0.2cm}

and 

\vspace*{0.2cm}

\begin{equation}
{\bf R}_{\langle {\A}\rangle}{\bf L}_{\A _{-4}}(\Ll _{-\omega _{\beta}}) =
{\bf R}_{\langle {\A}_{-2},{\A}_{-1},{\A}_0\rangle}(\Ll _{-\omega _{\beta}}) = 
\Psi _1^{\omega _{\alpha}}\otimes \Ll _{\rho}[-3]
\end{equation}

\vspace*{0.2cm}

To compute left dual bundle to $\pi _{\beta}^{\ast}\Uu _2$, we mutate it to the left through the subcategories $\A _{-3}$ and $\A_{-4}$, and then mutate the result to the right through the whole collection. The effect of the last action is described by Theorem \ref{th:mutations_of_block_collections}, and is equal to the right mutation of $\pi _{\beta}^{\ast}\Uu _2$ through the subcategories ${\A}_{-2},{\A}_{-1},{\A}_{0}$, i.e.
${\bf R}_{\langle {\A}_{-2},{\A}_{-1},{\A}_{0}\rangle}\pi _{\beta}^{\ast}\Uu _2$.

From (\ref{eq:U2_seq}) one finds that $\Hom ^{\bullet}(\Ll _{-\omega _{\alpha}},\pi _{\beta}^{\ast}\Uu _2)={\sf k}$ and $\Hom ^{\bullet}(\Ll _{-\omega _{\beta}},\pi _{\beta}^{\ast}\Uu _2)=\nabla _{\omega _{\alpha}}$. From (\ref{eq:spinor_ses}) one then finds ${\bf L}_{\langle \Ll _{-\omega _{\beta}}\rangle}\pi _{\beta}^{\ast}\Uu _2=\pi _{\beta}^{\ast}\Uu _2\otimes \Ll _{-\omega _{\beta}}[1]$. Further, $\Hom ^{\bullet}(\Ll _{-\omega _{\alpha}},\pi _{\beta}^{\ast}\Uu _2\otimes \Ll _{-\omega _{\beta}})={\sf k}[-1]$, and a unique non--trivial extension corresponds to the non--split short exact sequence

\vspace*{0.2cm}

\begin{equation}\label{eq:U_2toOmega_1extonSp_4/B}
0\rightarrow \pi _{\beta}^{\ast}\Uu _2\otimes \Ll _{\omega _{\alpha}}\rightarrow \Psi _1^{\omega _{\alpha}}\otimes \Ll _{\omega _{\alpha}}\rightarrow \Ll _{\omega _{\beta}}\rightarrow 0,
\end{equation}

\vspace*{0.2cm}

tensored with $\Ll _{-\rho}$. Indeed, concatenating two short exact sequences (\ref{eq:U2_seq}) and (\ref{eq:seqforNonSp_4/B}),
one obtains an element of the group  $\Ext ^2(\Ll _{\omega _{\beta}-\omega_{\alpha}},\Ll _{-\omega _{\alpha}})={\rm H}^2(\Ll _{-\omega _{\beta}})=0$, and the above sequence is the one that trivializes that element. Specifically, one has a commutative diagram:

\vspace{0.2cm}

\begin{figure}[H]
$$
\xymatrix{
& 0\ar@{->}[d] & 0 \ar@{->}[d] & & &\\
0\ar@{->}[r] & \Oo _{{\bf Sp}_4/{\bf B}} \ar@{->}[r]^{\quad \simeq}\ar@{->}[d] &\Oo _{{\bf Sp}_4/{\bf B}} \ar@{->}[r]\ar@{->}[d] &  0\ar@{->}[d] \\
0\ar@{->}[r] & \pi _{\beta}^{\ast}\Uu _2\otimes \Ll _{\omega _{\alpha}} \ar@{->}[r]\ar@{->}[d] & \Psi _1^{\omega _{\alpha}}\otimes \Ll _{\omega _{\alpha}}\ar@{->}[r]\ar@{->}[d] & \Ll _{\omega _{\beta}}\ar@{->}[d]\ar@{->}[r] & 0\\
0\ar@{->}[r] & \Ll _{2\omega _{\alpha}-\omega _{\beta}}\ar@{->}[r]\ar@{->}[d] & \pi _{\alpha}^{\ast}{\sf N}\otimes \Ll _{\omega _{\alpha}}\ar@{->}[r]\ar@{->}[d] & \Ll _{\omega _{\beta}}\ar@{->}[r] & 0\\
& 0 & 0 & & & \\
}
$$
  \end{figure}
  
 \vspace{0.2cm}

Thus, there is a distinguished triangle

\vspace*{0.2cm}

\begin{equation}
\dots \rightarrow \Psi _1^{\omega _{\alpha}}\otimes \Ll _{-\omega _{\beta}}\rightarrow \Ll _{-\omega _{\alpha}}\rightarrow 
\pi _{\beta}^{\ast}\Uu _2\otimes \Ll _{-\omega _{\beta}}[1]\rightarrow \dots ,
\end{equation}

\vspace*{0.2cm}

and one obtains ${\bf L}_{\langle \Ll _{-\omega _{\alpha}}\rangle}(\pi _{\beta}^{\ast}\Uu _2\otimes \Ll _{-\omega _{\beta}}[1])= \Psi _1^{\omega _{\alpha}}\otimes \Ll _{-\omega _{\beta}}[1]$. Putting all together, one has ${\bf L}_{\langle \Ll _{-\omega _{\alpha}}\rangle}({\bf L}_{\langle \Ll _{-\omega _{\beta}}\rangle}\pi _{\beta}^{\ast}\Uu _2) =  {\bf L}_{\langle {\A _{-3}}\rangle}\pi _{\beta}^{\ast}\Uu _2 =  \Psi _1^{\omega _{\alpha}}\otimes \Ll _{-\omega _{\beta}}[1]$. Finally, $\Hom ^{\bullet}(\Ll _{-\rho}, \Psi _1^{\omega _{\alpha}}\otimes \Ll _{-\omega _{\beta}})=
\Lambda ^2\nabla _{\omega _{\alpha}}^{\ast}$, and one finds:

\vspace*{0.2cm}

\begin{equation}
\dots \rightarrow {\bf L}_{\A _{-4}}( \Psi _1^{\omega _{\alpha}}\otimes \Ll _{-\omega _{\beta}})\rightarrow 
\Lambda ^2\nabla _{\omega _{\alpha}}^{\ast}\otimes \Ll _{-\rho}[1]\rightarrow  \Psi _1^{\omega _{\alpha}}\otimes \Ll _{-\omega _{\beta}}[1]\rightarrow \dots  ,
\end{equation}

\vspace*{0.2cm}

thus, ${\bf L}_{\langle \A _{-4}\rangle}( \Psi _1^{\omega _{\alpha}}\otimes \Ll _{-\omega _{\beta}}[1])= {\bf L}_{\langle \A _{-4}\rangle}{\bf L}_{\langle \A _{-3}\rangle}\pi _{\beta}^{\ast}\Uu _2 = {\bf L}_{\langle \A _{-4},\A _{-3}\rangle}\pi _{\beta}^{\ast}\Uu _2 = 
\Psi _2^{\omega _{\alpha}}\otimes \Ll _{-\rho}[2]$. Tensoring it with $\Ll _{2\rho}$ and applying Theorem \ref{th:mutations_of_block_collections}, one obtains:

\vspace*{0.2cm}

\begin{equation}
{\bf R}_{\langle {\A}_{-2},{\A}_{-1},{\A}_{0}\rangle}\pi _{\beta}^{\ast}\Uu _2 = \Psi _2^{\omega _{\alpha}}\otimes \Ll _{\rho}[-2].
\end{equation}

\vspace*{0.2cm}

Let us compute the right dual bundle to $\Psi _2^{\omega _{\alpha}}$. Similarly to the previous computation, we first compute the mutation of $\Psi _2^{\omega _{\alpha}}$ to the left through ${\A}_{-3}$ and ${\A}_{-4}$, and then mutate the result to the right through the whole collection. One computes $\Hom ^{\bullet}(\Ll _{-\omega _{\alpha}},\Psi _2^{\omega _{\alpha}})=\nabla _{\omega _{\alpha}}$, and $\Hom (\Ll _{-\omega _{\beta}},\Psi _2^{\omega _{\alpha}})={\sf k}$. Thus, one obtains ${\bf L}_{\langle \Ll _{-\omega _{\alpha}}\rangle}(\Psi _2^{\omega _{\alpha}})=\Ll _{-2\omega _{\alpha}}[1]$. Further, $\Hom ^{\bullet}(\Ll _{-\omega _{\beta}},\Ll _{-2\omega _{\alpha}})=
{\sf k}[-1]$, and a unique non--trivial extension corresponds to the non--split short exact sequence

\vspace*{0.2cm}

\begin{equation}
0\rightarrow \Ll _{\omega _{\beta}-\omega_{\alpha}}\rightarrow \pi _{\beta}^{\ast}\Uu _2^{\ast}\rightarrow \Ll _{\omega _{\alpha}}\rightarrow 0,
\end{equation}

\vspace*{0.2cm}

tensored with $\Ll _{-\rho}$. Thus, there is a distinguished triangle

\vspace*{0.2cm}

\begin{equation}
\dots \rightarrow \pi _{\beta}^{\ast}\Uu _2^{\ast}\otimes \Ll _{-\rho}\rightarrow \Ll _{-\omega _{\beta}} \rightarrow \Ll _{-2\omega_{\alpha}}[1]\rightarrow \dots ,
\end{equation}

\vspace*{0.2cm}

and ${\bf L}_{\langle \Ll _{-\omega _{\beta}}\rangle}\Ll _{-2\omega _{\alpha}}={\bf L}_{\langle \A_{-3}\rangle}(\Psi _2^{\omega _{\alpha}})= \pi _{\beta}^{\ast}\Uu _2^{\ast}\otimes \Ll _{-\rho}[1]$. Finally, $\Hom ^{\bullet}(\Ll _{-\rho},\pi _{\beta}^{\ast}\Uu _2^{\ast}\otimes \Ll _{-\rho})=\nabla _{\omega _{\alpha}}$, and there is a short exact sequence obtained from (\ref{eq:spinor_ses}) tensoring with $\Ll _{-\rho}$:

\vspace*{0.2cm}

\begin{equation}
0\rightarrow \pi _{\beta}^{\ast}\Uu _2\otimes \Ll _{-\rho}\rightarrow \nabla _{\omega _{\alpha}}\otimes \Ll _{-\rho}\rightarrow \pi _{\beta}^{\ast}\Uu _2^{\ast}\otimes \Ll _{-\rho}\rightarrow 0.
\end{equation}

\vspace*{0.2cm}

Thus, ${\bf L}_{\langle \Ll _{-\rho}\rangle}(\pi _{\beta}^{\ast}\Uu _2^{\ast}\otimes \Ll _{-\rho}[1])={\bf L}_{\langle \Ll _{-\rho}\rangle}{\bf L}_{\langle \A _{-3}\rangle}(\pi _{\beta}^{\ast}\Uu _2^{\ast}\otimes \Ll _{-\rho}[1]) =  \pi _{\beta}^{\ast}\Uu _2\otimes \Ll _{-\rho}[2]$. Theorem \ref{th:mutations_of_block_collections} then gives 

\vspace*{0.2cm}

\begin{equation}
{\bf R}_{\langle {\A}_{-2},{\A}_{-1},{\A}_{0}\rangle}(\Psi _2^{\omega _{\alpha}}) = \pi _{\beta}^{\ast}\Uu _2\otimes \Ll _{\rho}[-2].
\end{equation}
\end{proof}

\vspace*{0.2cm}

\begin{theorem}\label{th:FrobdecomSp_4/B}
The bundle ${\sf F _n}_{\ast}\Oo _{{\bf Sp}_4/{\bf B}}$ decomposes into the direct sum of vector bundles with indecomposable summands being isomorphic to:

\vspace*{0.2cm}

\begin{eqnarray}\label{eq:FrobrightdualcollSp_4_B}
&  \Oo _{{\bf Sp}_4/{\bf B}}, \quad  \Ll_{-\omega _{\alpha}}, \quad \Ll _{-\omega _{\beta}}, \quad 
\pi _{\beta}^{\ast}\Uu _2^{\ast}\otimes \Ll _{-\rho}, \quad (\Psi _2^{\omega _{\alpha}})^{\ast}\otimes \Ll _{-\rho}\\
&  (\Psi _1^{\omega _{\alpha}})^{\ast}\otimes \Ll _{-\rho}, \quad  (\Psi _1^{\omega _{\beta}})^{\ast}\otimes \Ll _{-\rho}, \quad \Ll _{-\rho}. \nonumber
\end{eqnarray}

\vspace*{0.2cm}

The multiplicity spaces at each indecomposable summand are isomorphic, respectively, to:

\vspace*{0.2cm}

\begin{eqnarray}
& {\sf k}, \quad \nabla _{p^n\omega _{\beta}}/\nabla _{\omega _{\beta}}^{[n]},
\quad \nabla _{p^n\omega _{\alpha}}/\nabla _{\omega _{\alpha}}^{[n]}, \quad 
\Delta _{(p^n-3)\omega _{\alpha}}/\nabla _{(p^n-4)\omega _{\alpha}+\omega _{\beta}}, \quad {\rm H}^2({\bf Sp}_4/{\bf B},{\sf F}_n^{\ast}\Psi _2^{\omega _{\alpha}}), \nonumber \\
&  \quad  \Delta _{(p^n-3)\omega _{\beta}}, \quad \Delta _{(p^n-3)\omega _{\alpha}}, \quad \Delta _{(p^n-2)\rho}.  
\end{eqnarray}

\vspace*{0.2cm}

\end{theorem}

\begin{proof}
Similar to the proof of Theorem \ref{th:FrobdecomSL_3/B}.
\end{proof}

As a corollary to Theorem \ref{th:FrobdecomSp_4/B}, one obtains another proof of the main result of \cite{SamD-affflag}:

\begin{corollary}
One has ${\rm H}^i({\bf Sp}_4/{\bf B},{\mathcal D}_{{\bf Sp}_4/{\bf B}})=0$ for $i>0$, and the flag variety ${\bf Sp}_4/{\bf B}$ is $\sf D$--affine.
\end{corollary}

\begin{proof}
The decomposition from Theorem \ref{th:FrobdecomSp_4/B} implies $\Ext ^i({\sf F _n}_{\ast}\Oo _{{\bf Sp}_4/{\bf B}},{\sf F _n}_{\ast}\Oo _{{\bf Sp}_4/{\bf B}})=0$ for $i>0$, since the indecomposable summands in the decomposition form an exceptional collection
by the very construction. Thus, $\Ext ^i({\sf F _n}_{\ast}\Oo _{{\bf Sp}_4/{\bf B}},{\sf F _n}_{\ast}\Oo _{{\bf Sp}_4/{\bf B}})=
{\rm H}^i({\bf Sp}_4/{\bf B},{\mathcal D}_{{\bf Sp}_4/{\bf B}}^{(n)})=0$ for $i>0$, and we conclude as in \cite{SamD-affflag}.
\end{proof}

\begin{remark}
{\rm Decomposition of ${\sf F _n}_{\ast}\Oo _{{\bf Sp}_4/{\bf B}}$ for $n=1$ was previously obtained in \cite{KanYe} (see also \cite{KanYe-loc}) using results of 
\cite{AK89} about the structure of ${\bf G}_1{\bf T}$--socle series of the induced ${\bf G}_1{\bf B}$--module $\hat{\nabla}(0)={\rm Ind}_{\bf B}^{{\bf G}_1{\bf B}}{\sf k}$.
Again, their methods are different from the one presented here.}
\end{remark}

\vspace{0.5cm}


\subsection{Type ${\bf G}_2$}


Recall the necessary facts about the flag variety ${\bf G}_2/{\bf B}$. We assume here that $p\neq 2,3$, thus working over ${\rm Spec}({\mathbb Z}[\frac{1}{6}])$.
Recall the necessary facts about the flag variety ${\bf G}_2/{\bf B}$.
The group ${\bf G}_2$ has two parabolic subgroups ${\bf P}_{\alpha}$ and ${\bf
  P}_{\beta}$ that correspond to the simple roots $\alpha$ and
$\beta$, the root $\beta$ being the long root. The homogeneous spaces ${\bf G}/{{\bf P}_{\alpha}}$ and  ${\bf G}/{{\bf P}_{\beta}}$ are isomorphic to the 5--dimensional variety ${\bf G}_2^{\rm ad}\subset \Pp (\nabla _{\omega _{\beta}})$ (the orbit of the highest vector in the adjoint representation of ${\bf G}_2$) and to the 5--dimensional quadric ${\sf Q}_5$, respectively. The Levi subgroups of ${\bf P}_{\alpha}$ and ${\bf  P}_{\beta}$ have a component isomorphic to ${\bf SL}_2$; its tautological representation in each case gives rise to a homogeneous rank 2 vector bundle on ${\bf G}/{{\bf P}_{\alpha}}$ (resp., on ${\bf G}/{{\bf P}_{\beta}}$). Denote $\pi _{\beta}$ and $\pi _{\alpha}$ the two projections of ${\bf G}_2/{\bf B}$ onto ${\bf G}_2^{\rm ad}\subset \Pp (\nabla _{\omega _{\beta}})$ and ${\sf Q}_5\subset \Pp (\nabla _{\omega _{\alpha}})$, respectively. More concretely, the projection $\pi _{\beta}$ is the projective bundle over ${\bf G}_2^{\rm ad}$ associated to a rank two vector bundle $\Uu _2$ over ${\bf G}_2^{\rm ad}\subset {\rm Gr}_{2,\Pp (\nabla _{\omega _{\alpha}})}$, and the projection $\pi _{\alpha}$ is the projective bundle associated to another rank two bundle $\mathcal K$ on ${\sf Q}_5$ with ${\rm det} \ \Kk = \Ll _{-3\omega_{\alpha}}$. 
There are short exact sequences: 

\vspace*{0.2cm}

\begin{equation}\label{eq:spinorsequence}
0\rightarrow {\Ss}\rightarrow {\sf U}\otimes \Oo _{{\sf
    Q}_n}\rightarrow {\Ss}^{\ast}\rightarrow 0,
\end{equation}

\vspace*{0.2cm}

where $\Ss$ the spinor bundle on ${\sf Q}_5$ and $\sf U$ is the spinor representation regarded as a representation of ${\bf G}_2$ via the restriction ${\bf G}_2\subset {\bf Spin}_7$, and 

\vspace*{0.2cm}

\begin{equation}\label{eq:relEulerseqforK}
0\rightarrow \Ll _{-\omega_{\beta}}\rightarrow  \pi _{\alpha}^{\ast}\Kk \rightarrow \Ll _{\omega _{\beta} - 3\omega _{\alpha}} \rightarrow 0.
\end{equation}

\vspace*{0.2cm}

One has ${\Ss}\otimes \Ll _{\omega _{\alpha}} = {\Ss}^{\ast}$.  The bundles $\Ss$ and $\Kk$ are related via a short exact sequence (Appendix B, \cite{Ku}):

\vspace*{0.2cm}

\begin{equation}\label{eq:seqforK}
0\rightarrow \Kk \rightarrow (\Psi _1^{\omega _{\alpha}})^{\ast}\otimes \Ll _{-\omega _{\alpha}}\rightarrow \Ss\rightarrow 0.
\end{equation}

\vspace*{0.2cm}

There is the tautological sequence for the bundle $\Uu _2$ on ${\bf G}_2^{\rm ad}$ obtained by the restriction of the tautological short exact sequence on ${\rm Gr}_{2,\Pp (\nabla _{\omega _{\alpha}})}$:

\vspace*{0.2cm}

\begin{equation}\label{eq:tautological_seq_U}
0\rightarrow \Uu _2\rightarrow \nabla _{\omega _{\alpha}}\otimes \Oo _{{\bf G}_2^{\rm ad}}\rightarrow 
\nabla _{\omega _{\alpha}}/\Uu _2\rightarrow 0.
\end{equation}

\vspace*{0.2cm}

Denote $\Uu _2^{\perp}: = (\nabla _{\omega _{\alpha}}/\Uu _2)^{\ast}$. Finally, there is a short exact sequence for $\Uu _2$:

\vspace*{0.2cm}

\begin{equation}\label{eq:seqforU}
0\rightarrow \Ll _{-\omega _{\alpha}}\rightarrow \pi _{\beta}^{\ast}\Uu _2\rightarrow \Ll _{\omega _{\alpha}-\omega _{\beta}}\rightarrow 0.
\end{equation}

\vspace*{0.2cm}

Consider a sequence $\A= \langle \A _i\rangle _{i=-6}^{i=0}$ of full triangulated subcategories of $\Dd ^b({\bf G}_2/{\bf B})$:

\begin{figure}[H]
\begin{equation}\label{eq:FrobexccollonG_2/B}
$$
\xymatrix{
 \A _{-6} & \A _{-5}& \A _{-4} & \A _{-3}& \A _{-2}& \A _{-1}& \A _0\\
 || & || & || & || & || & || & || \\
*++<10pt>[F]\txt{$\Ll _{-\rho}$}
&*++<10pt>[F]\txt{$\Ll _{-\omega _{\alpha}}$  \\ \\ $\Ll _{-\omega _{\beta}}$}
&*++<10pt>[F]\txt{$(\Psi _1^{\omega _{\alpha}})^{\ast}\otimes \Ll _{-\omega _{\alpha}}$  \\ \\ $\pi _{\beta}^{\ast}\Uu _2$}
&*++<10pt>[F]\txt{$\pi _{\alpha}^{\ast}\Ss$  \\ \\ $\pi _{\beta}^{\ast}\Ee$}
&*++<10pt>[F]\txt{$\Phi _2^{\omega _{\alpha}}$  \\ \\ $\pi _{\beta}^{\ast}\Uu _2^{\perp}$}
&*++<10pt>[F]\txt{$\Psi _1^{\omega _{\alpha}}$  \\ \\ $\Psi _1^{\omega _{\beta}}  $}
 &*++<10pt>[F]\txt{$\Oo _{{\bf G}_2/{\bf B}}$}
}
$$
\end{equation}
 \end{figure}

\vspace*{0.2cm}

Here $\Phi _2^{\omega _{\alpha}}={\bf L}_{\langle \Oo _{{\bf G}_2/{\bf B}},\Ll _{\omega _{\alpha}}\rangle}(\Ll _{2\omega _{\alpha}})$,\footnote{Recall that in \cite{Kap} $\Phi _2^{\omega _{\alpha}}$ is denoted $\Psi _2$.} and $\Ee={\pi _{\beta}}_{\ast}\Phi _2^{\omega _{\alpha}}$. To calculate $\Ee$ more explicitly, consdier the exact sequence defining $\Phi _2^{\omega _{\alpha}}={\bf L}_{\langle \Oo _{{\bf G}_2/{\bf B}},\Ll _{\omega _{\alpha}}\rangle}(\Ll _{2\omega _{\alpha}})$

\vspace*{0.2cm}

\begin{equation}\label{eq:defining_sequence_for_E}
0\rightarrow \Phi _2^{\omega _{\alpha}}\rightarrow \Oo _{{\bf G}_2/{\bf B}}\otimes (\Lambda ^2 \nabla _{\omega _{\alpha}}\oplus {\sf k})\rightarrow 
\nabla _{\omega _{\alpha}}\otimes \Ll _{\omega _{\alpha}}\rightarrow 
\Ll _{2\omega _{\alpha}}\rightarrow 0,
\end{equation}

\vspace*{0.2cm}

and apply ${\pi _{\beta}}_{\ast}$ to it:

\vspace*{0.2cm}

\begin{equation}\label{eq:Phi_2_sequence}
0\rightarrow {\pi _{\beta}}_{\ast}\Phi _2^{\omega _{\alpha}}\rightarrow \Oo _{{\bf G}_2^{\rm ad}}\otimes (\Lambda ^2 \nabla _{\omega _{\alpha}}\oplus {\sf k})\rightarrow \nabla _{\omega _{\alpha}}\otimes \Uu _2^{\ast}\rightarrow 
{\sf S}^2\Uu _2^{\ast}\rightarrow 0.
\end{equation}

\vspace*{0.2cm}

Taking into account the dual of the second exterior power of tautological short exact sequence (\ref{eq:tautological_seq_U})

\vspace*{0.2cm}

\begin{equation}\label{eq:wedge_2_tautU}
0\rightarrow \Lambda ^2\Uu _2^{\perp}\rightarrow \Oo _{{\bf G}_2^{\rm ad}}\otimes \Lambda ^2 \nabla _{\omega _{\alpha}}\rightarrow \nabla _{\omega _{\alpha}}\otimes \Uu _2^{\ast}\rightarrow 
{\sf S}^2\Uu _2^{\ast}\rightarrow 0,
\end{equation}

\vspace*{0.2cm}

we obtain $\Ee$ as a unique non--trivial extension on ${\bf G}_2^{\rm ad}$

\vspace*{0.2cm}

\begin{equation}\label{eq:defining_sequence_for_E}
0\rightarrow \Lambda ^2\Uu _2^{\perp}\rightarrow \Ee \rightarrow \Oo _{{\bf G}_2^{\rm ad}}\rightarrow 0.
\end{equation}

\vspace*{0.2cm}

\begin{theorem}\label{lem:G_2-lemma}
The sequence $\A$ is a semiorthogonal decomposition of $\Dd ^b({\bf G}_2/{\bf B})$. The right dual decomposition $\C$ with respect to (\ref{eq:FrobexccollonG_2/B}) consists of the following subcategories:

\vspace*{0.2cm}

\begin{figure}[H]
\begin{equation}\label{eq:FrobrightdualexccollonG_2/B}
$$
\xymatrix{
{\C}_{0} & {\C}_1& {\C}_2 & {\C}_3 & {\C}_4& {\C}_5 &{\C}_6\\
|| & || & || & || & || & || & || \\
*++<10pt>[F]\txt{$\Oo _{{\bf G}_2/{\bf B}}$}
&*++<10pt>[F]\txt{$\Ll _{\omega _{\alpha}}$  \\ \\ $\Ll _{\omega _{\beta}}$}
&*++<10pt>[F]\txt{$\pi _{\beta}^{\ast}\Uu _2^{\ast}\otimes \Ll _{\omega _{\alpha}}$  \\ \\ $(\Psi _1^{\omega _{\alpha}})^{\ast}\otimes \Ll _{\omega _{\beta}}$}
&*++<10pt>[F]\txt{$\pi _{\beta}^{\ast}\Ee \otimes \Ll _{\rho}$ \\ \\ $\pi _{\alpha}^{\ast}\Ss \otimes \Ll _{\rho}$}
&*++<10pt>[F]\txt{$\pi _{\beta}^{\ast}\Uu _2^{\perp}\otimes \Ll _{\rho}$ \\ \\ $\Phi _2^{\omega _{\alpha}}\otimes \Ll _{\rho}$}
&*++<10pt>[F]\txt{$\Psi _1^{\omega _{\beta}}\otimes \Ll _{\rho}$  \\ \\ $\Psi _1^{\omega _{\alpha}}\otimes \Ll _{\rho}$}
 &*++<10pt>[F]\txt{$\Ll _{\rho}$}
}
$$
\end{equation}
\end{figure}

\vspace*{0.2cm}

The bundles in the subcategory $\C _i$ for $-6 \leq i\leq 0$ are shifted by $[i]$ .

\end{theorem}

\begin{proof}
The necessary orthogonalities in (\ref{eq:FrobexccollonG_2/B}) and the fact that $\A$ generates $\Dd ^b({\bf G}_2/{\bf B})$ are verified directly using the vanishing theorems of Section \ref{sec:Cohomology_on_G/B} and Theorem \ref{th:Orvlovth}. Tensoring $\mathcal A$ with $\Ll _{\rho}$ and flipping the two bundles in each block of $\mathcal A$, one obtains the decomposition $\C$; using Proposition \ref{prop:dual_exc_coll_characterization} that characterizes the right dual decomposition, one ensures that $\C$ is the right dual to $\A$.

\end{proof}

\end{document}